\documentclass[a4paper, reqno,12pt]{amsart}

\usepackage{tikz, tikz-3dplot}
\usepackage{amsmath,amsthm,amssymb,graphics}
\usepackage[utf8]{inputenc}
\usepackage[english]{babel}
\usepackage[colorlinks=true, allcolors=blue]{hyperref}
\usepackage{slashbox}
\usepackage[title]{appendix}
\numberwithin{equation}{section}

\theoremstyle{plain}

\newtheorem{te}{Theorem}[section]
\newtheorem{coro}[te]{Corollary}
\newtheorem{prop}[te]{Proposition}

\newtheorem{lem}[te]{Lemma}

\newtheorem*{ack*}{Acknowledgment}
\theoremstyle{remark}

\newcommand{\dsum}{\displaystyle\sum}
\newcommand{\dint}{\displaystyle\int}

\def\y{{\bf y}}

\def\0{{\bf 0}}

\def\R{{\mathbb R}}

\def\N{{\mathbb N}}

\def\Z{{\mathbb Z}}

\def\nint{\mathop{\diagup\kern-13.0pt\int}}

\begin{document}

\author{Kiseok Yeon}
\email{kyeon@purdue.edu}
\address{Department of Mathematics, Purdue University, 150 N. University Street, West Lafayette, IN 47907-2067, USA}

\title[Small fractional parts of binary forms]{Small fractional parts of binary forms}
\maketitle

\begin{abstract}
    We obtain bounds on fractional parts of binary forms of the shape $$\Psi(x,y)=\alpha_k x^k+\alpha_l x^ly^{k-l}+\alpha_{l-1}x^{l-1}y^{k-l+1}+\cdots+\alpha_0 y^k$$ with $\alpha_k,\alpha_l,\ldots,\alpha_0\in\R$ and $l\leq k-2.$ By exploiting recent progress on Vinogradov's mean value theorem and earlier work on exponential sums over smooth numbers, we derive estimates superior to those obtained hitherto for the best exponent $\sigma$, depending on $k$ and $l,$ such that
    \begin{equation*}
        \min_{\substack{0\leq x,y\leq X\\(x,y)\neq (0,0)}}\|\Psi(x,y)\|\leq X^{-\sigma+\epsilon}.
    \end{equation*}

\end{abstract}

\section{Introduction}

  In 1914, Hardy and Littlewood [$\ref{ref15}$, p 172] posed the question as to whether, when $\alpha\in\R$, $k\in\N$ and $\epsilon>0$ is any positive number, there exists $\sigma>0$ not depending on $\alpha$ such that 
\begin{equation}\label{1}
    \min_{1\leq x\leq X}\|\alpha x^k\|\leq X^{-\sigma+\epsilon},
\end{equation}
 provided that $X$ is sufficiently large in terms of $k$ and $\epsilon$. Here, $\|\cdot\|$ denotes the distance to the nearest integer. Vinogradov [$\ref{ref17}$] gave a positive answer to this question by providing  a specific exponent $\sigma=k/(2^{k-1}k+1)$. This was quantitatively improved by Heilbronn [$\ref{ref8}$], who replaced the exponent $\sigma=2/5$ by $1/2$ in the case $k=2$. Danicic [$\ref{ref6}$] extended this conclusion to $\sigma=1/2^{k-1}$ for all $k\in \N$. Subsequently, the exponent $1/2$ in the case $k=2$ was improved to $\sigma=4/7$ by Zaharescu [$\ref{ref19}$]. Sharper results for $k\geq 6$ depend on recent progress Vinogradov's mean value theorem, and thus Baker $[\ref{ref3}]$ proved that the exponent $\sigma$ can be replaced by $\sigma=1/(k(k-1))$. Furthermore, Wooley [$\ref{ref12}$] showed that the exponent $\sigma$ can be improved to $ \sigma= 1/(k\log k+O(k\log\log k))$. However, it is conjectured in [$\ref{ref1}$] that the exponent $\sigma$ should be improvable to $1.$ Unfortunately, current techniques do not seem capable of achieving this conjecture.

Turning our attention to problems in many variables, one is led to ask analogous questions regarding the distribution mod $1$ of polynomials $f(x_1,\ldots,x_s)\in \R[x_1,\ldots,x_s]$. In [$\ref{ref1}$, Theorem 10.2], Baker proved that when $\alpha_1,\ldots,\alpha_s\in \R$ and $k\in \N$, there exists $s_0$  such that whenever $s\geq s_0$, one has $$\min_{\substack{0\leq \boldsymbol{x}\leq X\\\boldsymbol{x}\neq\boldsymbol{0}}}\|\alpha_1 x_1^k+\cdots+\alpha_s x_s^k\|\leq X^{-k+\epsilon}.$$ 
As a quantitative generalization, Baker [$\ref{ref2}$,$\ref{ref3}$] provided an explicit exponent $\sigma(s,k)$ such that $$\min_{\substack{0\leq \boldsymbol{x}\leq X\\\boldsymbol{x}\neq\boldsymbol{0}}}\|\alpha_1x_1^k+\cdots+\alpha_sx_s^k\|\leq X^{-\sigma(s,k)+\epsilon}.$$
For polynomials more general than additive forms, Schmidt [$\ref{ref20}$] made progress on the analogous problem regarding forms of odd degree $k$. Specifically, for given $E>0$,  there exist $s_0=s_0(E,k)$ such that whenever $s>s_0$, one has
   \begin{equation}\label{2'}
       \min_{\substack{0\leq |\boldsymbol{x}|\leq X\\\boldsymbol{x}\neq \boldsymbol{0}}}\|f(\boldsymbol{x})\|< X^{-E}.
   \end{equation}
In the special case of cubic form $f(\underline{\mathbf{x}})$, Dietmann [$\ref{ref4}$] provided explicit formula $\sigma(s)$ such that
  \begin{equation*}
       \min_{\substack{0\leq |\boldsymbol{x}|\leq X\\\boldsymbol{x}\neq \boldsymbol{0}}}\|f(\boldsymbol{x})\|< X^{-\sigma(s)}.
   \end{equation*}

In this paper, we seek to make progress on making bounds of the shape $(\ref{2'})$ explicit, concentrating for the present on the situation with $s=2.$
Here, one may write $f(x_1,x_2)=\alpha_kx_1^k+\alpha_{k-1}x_1^{k-1}x_2+\cdots+\alpha_0x_2^k$ with $\alpha_k,\ldots,\alpha_0\in \R$. Hitherto, the only general strategy available has been a trivial one that reduces this problem in two variables, by specializing $x_2=0$ for example, to a problem in one variable. In this direction, Wooley [$\ref{ref13}$, Theorem 4] obtained a nontrivial exponent $\sigma$ by restricting the polynomial to have the shape $\alpha\Phi(x_1,x_2)$ where $\Phi(x_1,x_2)\in \Z[x_1,x_2]$ is a binary form of degree $k$ exceeding $1$.
The main difficulties of the problem regarding general binary forms in $\R[x_1,x_2]$ arise from our relatively poor understanding concerning how the monomials $\alpha_lx_1^lx_2^{k-l}$, with $0\leq l\leq k$, combine to influence the fractional parts of $f(x_1,x_2)$. Our goal in this paper is to explore approaches to overcoming these difficulties and go beyond the trivial approach hitherto applied, concentrating on the class of binary forms. Our hope is that the arguments described here may be useful in studying fractional parts of polynomials in many variables in general.

 
\bigskip


Our first theorem provides bounds for binary forms of degree $k$ useful for small $k$. 

\begin{te}
Let $k$ and $l$ be natural numbers satisfying $1\leq l\leq k-2,$ and consider real numbers $\alpha_k$, $\alpha_j$ ($1\leq j\leq l$). Then, for any $\epsilon>0$, there exists a real number $X(k,\epsilon)$ such that whenever $X\geq X(k,\epsilon)$ one has
\begin{equation}\label{2}
    \min_{\substack{0\leq x,y\leq X\\(x,y)\neq(0,0)}}\|\alpha_k x^k+\alpha_l x^ly^{k-l}+\alpha_{l-1}x^{l-1}y^{k-l+1}+\cdots+\alpha_0 y^k\|< X^{-\sigma+\epsilon},
\end{equation}
where $$\sigma=\frac{l+2}{l+1}2^{1-k}.$$
\end{te}

For comparison with earlier results, by incorporating the conclusion of Danicic [$\ref{ref6}$] with the trivial strategy mentioned above, one obtains
\begin{equation*}
     \min_{\substack{0\leq x,y\leq X\\(x,y)\neq(0,0)}}\|\alpha_k x^k+\alpha_l x^ly^{k-l}+\alpha_{l-1}x^{l-1}y^{k-l+1}+\cdots+\alpha_0 y^k\|\leq \min_{1\leq x\leq X}\|\alpha_k x^k\|\leq X^{-2^{1-k}+\epsilon}
\end{equation*}
The conclusion of Theorem 1.1 improves on the exponent $2^{1-k}$ here by a factor which may be as large as $\frac{3}{2}$ in the case $l=1.$


\begin{coro}
Let $k$ be a natural number, and consider real numbers $\alpha_k, \alpha_j\ (0\leq j\leq k-2)$. Let $\varphi(x,y)=c_kx^k+c_{k-1}x^{k-1}y+\cdots+c_1xy^{k-1}+c_0y^k$ with $c_j\in \Z.$ Then, for any $\epsilon>0$, there exist a real number $X(\epsilon,c_k,c_{k-1})$ such that whenever $X\geq X(\epsilon,c_k,c_{k-1})$ one has
\begin{equation*}
    \min_{\substack{0\leq |x|,|y|\leq X\\(x,y)\neq (0,0)}}\|\alpha_k \varphi(x,y)+\alpha_{k-2} x^{k-2}y^2+\cdots+\alpha_0 y^k\|< X^{-\sigma+\epsilon},
\end{equation*}
where $\sigma=(k/(k-1))2^{1-k}.$
\end{coro}

In [$\ref{ref4}$], Dietmann recorded that for sufficiently large $X,$ one has
\begin{equation*}
   \min_{\substack{0\leq |x|,|y|\leq X\\(x,y)\neq (0,0)}}\|\alpha x^3+\beta x^2y+\gamma xy^2+\delta y^3\|< X^{-1/4+\epsilon},
\end{equation*}
which is obtained by specializing $y=0$ and applying the earlier work of Danicic $[\ref{ref6}]$. In Corollary 1.2 with $k=3$, we obtain the superior exponent $\sigma=3/8$ by imposing the condition that $\alpha$ and $\beta$ have a rational ratio. By further restricting $\alpha,\beta,\gamma,\delta$ to have rational ratio to each other, Wooley [$\ref{ref13}$, Theorem 4] proved that this can be improved to $\sigma=1/2.$

\bigskip

Our next two theorems provide bounds useful when $k$ is larger by exploiting recent progress on Vinogradov's mean value theorem and earlier work applying smooth numbers. The statements of these results is facilitated by writing $$\rho(k,l)=2^2(k-1)+2^3(k-2)+\cdots+2^{l+1}(k-l),$$
and 
\begin{align}
    &l_0= \max\{l\in \N|\ 7\rho(k,l)\leq k(k-1)\}\\
    &l_1= \max\{l\in \N|\ 7\rho(k,l)\leq k\log k\},
\end{align}
where the number 7 here may be replaced by smaller positive numbers.
\begin{te}
Let $k$ and $l$ be natural numbers satisfying $1\leq l\leq l_0$ and $k\geq 3$, and consider real numbers $\alpha_k$ and $\alpha_j$ $(1\leq j\leq l)$. Then, for any $\epsilon>0$, there exists a real number $X(k,\epsilon)$ such that whenever $X\geq X(k,\epsilon)$ one has
\begin{equation}\label{4}
    \min_{\substack{0\leq x,y\leq X\\(x,y)\neq(0,0)}}\|\alpha_k x^k+\alpha_l x^ly^{k-l}+\alpha_{l-1}x^{l-1}y^{k-l+1}+\cdots+\alpha_0 y^k\|< X^{-\sigma+\epsilon},
\end{equation}
where $$\sigma=\frac{2}{k(k-1)+\rho(k,l)}$$
\end{te}

For comparison with earlier results, by combining the earlier work of Baker [$\ref{ref3}$, Theorem 3] and the trivial treatment mentioned above, one obtains
\begin{equation*}
   \min_{\substack{0\leq x,y\leq X\\(x,y)\neq(0,0)}}\|\alpha_k x^k+\alpha_l x^ly^{k-l}+\alpha_{l-1}x^{l-1}y^{k-l+1}+\cdots+\alpha_0 y^k\|<\min_{\substack{1\leq x\leq X}}\|\alpha x^k\|\leq X^{-\sigma+\epsilon},
\end{equation*}
where $\sigma=1/(k(k-1)).$ Notice that for all $l$ with $1\leq l\leq l_0,$ the conclusion of Theorem 1.3 is superior to this earlier result stemming from Baker [$\ref{ref3}$, Theorem 3]. For instance, in the case $l=1$, the earlier result, which is $\sigma=1/(k(k-1))$, can be improved to 
$$\sigma=\frac{2}{(k+2)(k-1)}.$$

\bigskip

\begin{te}
Let $k$ and $l$ be natural numbers satisfying $1\leq l\leq l_1$ and $k\geq 3$, and consider real numbers $\alpha_k$ and $\alpha_j$ $(1\leq j\leq l)$. Then, for any $\epsilon>0$, there exist a real number $X(k,\epsilon)$ such that whenever $X\geq X(k,\epsilon)$ one has 
\begin{equation}
     \min_{\substack{0\leq x,y\leq X\\(x,y)\neq(0,0)}}\|\alpha_k x^k+\alpha_l x^ly^{k-l}+\alpha_{l-1}x^{l-1}y^{k-l+1}+\cdots+\alpha_0 y^k\|< X^{-\sigma+\epsilon},
\end{equation}
where $$\sigma=\frac{2}{k\log k+\rho(k,l)+Ck\log\log k}$$
in which $C$ is a positive constant which may not depend on $k$.
\end{te}

For comparison with earlier results, by combining the earlier work of Wooley $[\ref{ref18}$, Theorem 1.2] and the trivial method mentioned above, one obtains
\begin{equation*}
   \min_{\substack{0\leq x,y\leq X\\(x,y)\neq(0,0)}}\|\alpha_k x^k+\alpha_l x^ly^{k-l}+\alpha_{l-1}x^{l-1}y^{k-l+1}+\cdots+\alpha_0 y^k\|<\min_{\substack{1\leq x\leq X}}\|\alpha x^k\|\leq X^{-\sigma+\epsilon},
\end{equation*}
where $\sigma=1/(k\log k+Ck\log\log k)$ in which $C$ is a positive constant which may not depend on $k$. Notice that for all $l$ with $1\leq l\leq l_1,$ the conclusion of Theorem 1.4 is superior to this earlier result stemming from Wooley [$\ref{ref18}$, Theorem 1.2].

The method underlying the proof of our main theorems is to exploit inductive arguments based on a classical argument widely used in studying fractional parts of polynomials. However, it is complicated by the need to obtain quantitative inequalities useful for each of the inductive steps. Thus, these and other technical complications may obscure the key ideas of our arguments, and so we sketch this inductive argument in Section 2. Then, in Section 3, we provide key lemmas about the quantitative inequalities required for the success of our arguments. In Sections 4, 5 and 6, by exploiting these key lemmas and the inductive arguments, we provide the proof of Theorem 1.1, 1.3 and 1.4, respectively. Throughout this paper, we use $\gg$ and $\ll$ to denote Vinogradov's well-known notation, and write $e(z)$ for $e^{2\pi iz}$. We adopt the convention that whenever $\epsilon$ and $\delta$ appear in a statement, then the statement holds for each $\epsilon,\delta>0$, with implicit constants depending on $\epsilon$ and $\delta$, respectively.

\section*{Acknowledgement}
The author acknowledges support from NSF grant DMS-2001549 under the supervision of Trevor Wooley. The author is grateful for support from Purdue University. The author would like to thank Trevor Wooley for careful reading and helpful comments.


\bigskip

\section{Outline of the proof}

In this paper, we mainly follow a classical argument relating fractional parts of polynomials to lower bounds of associated exponential sums [$\ref{ref1}$, Theorem 2.2]. We record here this theorem.
\begin{lem}
Let $x_1,\ldots,x_N$ be real numbers. Suppose that $\|x_n\|\geq H^{-1}$ for every $n$ with $1\leq n\leq X$. Then, 
\begin{equation*}
    \dsum_{1\leq h\leq H}\biggl|\dsum_{n=1}^Ne(hx_n)\biggl|\geq N/6
\end{equation*}
\begin{proof}
See [$\ref{ref1}$, Theorem 2.2].
\end{proof}
\end{lem}

Suppose that for all $x,y$ with $x,y\in [1,X]\cap \Z$ one has
\begin{equation}\label{8''}
    \|\alpha_k x^k+\alpha_{l} x^ly^{k-l}+\alpha_{l-1}x^{l-1}y^{k-l+1}+\cdots+\alpha_0 y^k\| > 1/H.
\end{equation}
Then, by Lemma 2.1, one obtains
\begin{equation}\label{9'}
    \dsum_{1\leq h\leq H}|\dsum_{1\leq x,y\leq X}e(h(\alpha_k x^k+\alpha_{l} x^ly^{k-l}+\alpha_{l-1}x^{l-1}y^{k-l+1}+\cdots+\alpha_0 y^k))|\gg X^{2}.
\end{equation}

Our first goal is to obtain information about Diophantine approximations to $\alpha_l$ from the lower bound ($\ref{9'}$). In Section 3, we shall obtain upper bounds for the exponential sum on the left hand side of ($\ref{9'}$) in terms of $\alpha_k$ and $\alpha_l$.  
Combining these upper bound with the lower bound ($\ref{9'}$), one obtains the information concerning $\alpha_{l}.$

The next goal is to reduce the problem of bounding the fractional parts on the left hand side of ($\ref{8''}$) to the corresponding diagonal problem through inductive arguments. Note that by the triangle inequality, one has
\begin{multline}\label{8}
    \|\alpha_k x^k+\alpha_{l} x^ly^{k-l}+\alpha_{l-1}x^{l-1}y^{k-l+1}+\cdots+\alpha_0 y^k\|\\\leq \|\alpha_k x^k+\alpha_{l-1}x^{l-1}y^{k-l+1}+\cdots+\alpha_0 y^k\|+\|\alpha_l x^ly^{k-l}\|.
\end{multline}
By putting $y=qy'$ where $q$ is the denominator associated with a rational approximation to $\alpha_l,$ and again applying the triangle inequality, we see that the right hand side in ($\ref{8}$) is bounded above by
\begin{equation}\label{9}
 \|\alpha_k x^k+\alpha_{l-1}x^{l-1}(qy')^{k-l+1}+\cdots+\alpha_0 (qy')^k\|+x^l(qy')^{k-l-1}y'\|q\alpha_l\|.
\end{equation}
From the information about $\alpha_l$ obtained in the first step, one finds that $\|q\alpha_l\|$ is small enough so that the second term in $(\ref{9})$ is $\leq 1/(2H)$. Thus, the inequality ($\ref{8''}$) implies that for all $x,y'$ one has
\begin{equation*}
    \|\alpha_k x^k+\alpha_{l-1}x^{l-1}(qy')^{k-l+1}+\cdots+\alpha_0 (qy')^k\|\gg 1/H.
\end{equation*}
By applying this argument inductively, we ultimately infer from ($\ref{8''}$) that for all $x$ and $y_0$ one has
\begin{equation*}
    \|\alpha_k x^k+\alpha_0' y_0^k\|\gg 1/H
\end{equation*}
for all $x$ and $y_0.$
Experts will recognize that the diagonal polynomial here is accessible to classical methods.




\bigskip

\section{Exponential sums in two variables}

Let $A$ and $B$ be subsets of $[1,X]\cap \Z.$ Suppose that $$\|\alpha_k x^k+\alpha_{l} x^ly^{k-l}+\alpha_{l-1}x^{l-1}y^{k-l+1}+\cdots+\alpha_0 y^k\|\geq H^{-1}$$ for all $x\in A,\ y\in B.$ Then, by Lemma 2.1, one finds that
 \begin{equation}\label{11'}
     \dsum_{h\leq H}\biggl|\dsum_{\substack{x\in A\\y\in B}}e(h(\alpha_k x^k+\alpha_{l} x^ly^{k-l}+\alpha_{l-1}x^{l-1}y^{k-l+1}+\cdots+\alpha_0 y^k))\biggr|\gg |A||B|,
\end{equation}
where $|A|$ and $|B|$ are cardinalities of the sets $A$ and $B$, respectively.

Our goal in this section is to obtain upper bounds for the exponential sum above in terms of $\alpha_k$ and $\alpha_l$. To obtain these upper bounds, we introduce a variant of conventional Weyl differencing applicable to exponential sums in many variables, and exploit an argument used in earlier works on exponential sums in many variables (See, for example, [$\ref{ref10}$, Lemma 4.3]). To facilitate concision throughout this section, we define $g(x,y)=g(x,y;\boldsymbol{\alpha})\in \R[x,y]$ by
$$g(x,y)=\alpha_k x^k+\alpha_l x^ly^{k-l}+\alpha_{l-1}x^{l-1}y^{k-l+1}+\cdots+\alpha_0 y^k$$
with $l\leq k-2.$ Furthermore, we define differencing operators $\Delta_1^x$ and $\Delta_1^{y}$ by 
$$\Delta_1^x(g(x,y);h)=g(x+h,y)-g(x,y)$$
$$\Delta_1^y(g(x,y);h)=g(x,y+h)-g(x,y),$$
and so we define $\Delta_j^x$ and $\Delta_j^y$ for $j\geq 2$ recursively by means of the relations
$$\Delta_j^x(g(x,y);\mathbf{h})=\Delta_j^x(g(x,y);h_1,...,h_j)=\Delta_1^x(\Delta_{j-1}^x(g(x,y);h_1,...,h_{j-1});h_j)$$
$$ \Delta_j^y(g(x,y);\mathbf{h})=\Delta_j^y(g(x,y);h_1,...,h_j)=\Delta_1^y(\Delta_{j-1}^y(g(x,y);h_1,...,h_{j-1});h_j).$$
 
\bigskip

The following lemma provides an upper bound for the exponential sum in ($\ref{11'}$) with $A=[1,X]\cap \Z$ and $B=[1,Y]\cap \Z$. This upper bound is useful for small $k.$ In advance of the statement of this lemma, we define the exponential sum $T(\boldsymbol{\alpha})=T(\boldsymbol{\alpha}; H,X,Y)$ by 

$$T(\boldsymbol{\alpha})=\dsum_{1\leq h\leq H}\biggl|\dsum_{1\leq x\leq X}\dsum_{1\leq y\leq Y}e(hg(x,y))\biggr|.$$

\begin{lem}
Let $k$ and $l$ be natural numbers satisfying $1\leq l\leq k-2$ and $k\geq 3,$ and consider real numbers $\alpha_k$ and $\alpha_j\ (0\leq j\leq l)$. Suppose that $\alpha_k$ and $\alpha_{l}$ have rational approximations $|\alpha_k-a_1/q_1|\leq q_1^{-2}$ and $|\alpha_{l}-a_2/q_2|<q_2^{-2}$ with $(q_1,a_1)=(q_2,a_2)=1,$ respectively. Suppose that $X$ and $H$ are positive real numbers sufficiently large in terms of $k$. Then, for each $\epsilon>0$, one has
\begin{equation}\label{11}
   T(\boldsymbol{\alpha})\ll H(XY)^{1+\epsilon}\left(\frac{1}{q_1}+\frac{1}{X}+\frac{q_1}{HX^k}\right)^{2^{1-k}}\left(\frac{1}{q_2}+\frac{1}{Y}+\frac{q_2}{HX^{l}Y^{k-l}}\right)^{2^{1-k}}
\end{equation}
\end{lem}

\begin{proof}

By H\"older's inequality, one obtains 
\begin{equation*}
    T(\boldsymbol{\alpha})\ll \biggl(\dsum_{1\leq h\leq H}\dsum_{1\leq y\leq Y} 1\biggr)^{1-2^{-l}}\biggl(\dsum_{1\leq h\leq H}\dsum_{1\leq y\leq Y}\biggl|\dsum_{1\leq x\leq X}e(hg(x,y))\biggr|^{2^l}\biggr)^{2^{-l}}.
\end{equation*}
Then by Weyl differencing (see, for example, the proof of Vaughan [$\ref{ref11}$, Lemma 2.4]), we obtain
\begin{equation*}
    \begin{aligned}
          T(\boldsymbol{\alpha})\ll (HY)^{1-2^{-l}}\biggl(X^{2^l-(l+1)}\dsum_{1\leq h\leq H}\dsum_{1\leq y\leq Y}\dsum_{|h_1|\leq X}\cdots\dsum_{|h_l|}\dsum_{x\in I_l{(\boldsymbol{h}})}e(h\Delta_l^x(g(x,y),\mathbf{h}))\biggr)^{2^{-l}}
    \end{aligned}
\end{equation*}
where $I_l({\boldsymbol{h}})$ is an interval of integers contained in $[-X,X].$ Note here that $$\Delta_l^x(g(x,y);\mathbf{h})=\Delta_l^x(\alpha_k x^k)+\Delta_l^x(\alpha_lx^ly^{k-l})$$ and that $\Delta_l^x(\alpha_lx^ly^{k-l})$ does not depend on $x.$ Hence, one obtains
\begin{equation}\label{eq3.3}
    \begin{aligned}
          T(\boldsymbol{\alpha})\ll (HY)^{1-2^{-l}}X^{1-(l+1)2^{-l}}V(\alpha_k,\alpha_l)^{2^{-l}},
    \end{aligned}
\end{equation}
where $$V(\alpha_k,\alpha_l)=\dsum_{1\leq h\leq H}\dsum_{|h_1|\leq X}\cdots\dsum_{|h_l|\leq X}|V_1(\alpha_k;h,\mathbf{h})||V_2(\alpha_l;h,\mathbf{h})|$$ in which
$$V_1(\alpha_k;h,\mathbf{h})=\dsum_{x\in I_l}e(h\Delta_l^x(\alpha_k x^k,\mathbf{h}))$$ and
$$V_2(\alpha_l;h,\mathbf{h})=\dsum_{1\leq y\leq Y}e(h\Delta_l^x(\alpha_l x^ly^{k-l},\mathbf{h})).$$

By applying H\"older's inequality again,
\begin{equation}\label{3.3}
    \begin{aligned}
          V(\alpha_k,\alpha_l) \ll \biggl(\dsum_{1\leq h\leq H}\dsum_{|h_1|\leq X}\cdots\dsum_{|h_l|\leq X}1\biggr)^{1-2^{2-k+l}}W_1(\alpha_k)^{2^{1-k+l}}W_2(\alpha_l)^{2^{1-k+l}}
    \end{aligned}
\end{equation}
where 
$$W_1(\alpha_k)=\dsum_{1\leq h\leq H}\dsum_{|h_1|\leq X}\cdots\dsum_{|h_l|\leq X}|V_1(\alpha_k,h,\mathbf{h})|^{2^{k-l-1}},$$
$$W_2(\alpha_l)=\dsum_{1\leq h\leq H}\dsum_{|h_1|\leq X}\cdots\dsum_{|h_l|\leq X}|V_2(\alpha_l;h,\mathbf{h})|^{2^{k-l-1}}.$$

By applying classical Weyl differencing arguments, one has
\begin{equation*}
    \begin{aligned}
          W_1(\alpha_k)\ll X^{2^{k-l-1}-(k-l)}\dsum_{1\leq h\leq H}\dsum_{|h_1|\leq X}\cdots\dsum_{|h_{k-1}|\leq X}\biggl|\dsum_{x\in I_{k-1}(\boldsymbol{h})}e(c_khh_1\cdots h_{k-1}x)\biggr|
    \end{aligned}
\end{equation*}
where $c_k$ is a constant depending on $k,$ and $I_{k-1}(\boldsymbol{h})$ is an interval of integers contained in $[-X,X]$. Thus, for each $\epsilon>0,$ one has

\begin{equation*}
    \begin{aligned}
          W_1(\alpha_k)\ll HX^{2^{k-l-1}+l+\epsilon}\left(\frac{1}{q_1}+\frac{1}{X}+\frac{q_1}{HX^k}\right).
    \end{aligned}
\end{equation*}
Furthermore, one has
\begin{equation*}
    \begin{aligned}
          W_2(\alpha_l)
          \ll Y^{2^{k-l-1}-(k-l)}\dsum_{1\leq h\leq H}\dsum_{|h_1|\leq X}\cdots\dsum_{|h_l|\leq X}\dsum_{|h_{l+1}|\leq Y}\cdots\dsum_{|h_{k-1}|\leq Y}\biggl|\dsum_{y\in J_{k-1}(\boldsymbol{h})}e(d_khh_1\cdots h_{k-1}y)\biggr|
    \end{aligned}
\end{equation*}
where 
$d_k$ is a constant depending on $k,$ and $J_{k-1}(\boldsymbol{h})$ is an interval of integers contained in $[-Y,Y].$ Hence, for each $\epsilon,$ one has
$$W_2(\alpha_l)\ll HX^lY^{2^{k-l-1}+\epsilon}\left(\frac{1}{q_2}+\frac{1}{Y}+\frac{q_2}{HX^lY^{k-l}}\right).$$
Substituting these estimates for $W_1(\alpha_k)$ and $W_2(\alpha_l)$ into ($\ref{3.3}$) and from there into ($\ref{eq3.3}$), one concludes that
\begin{equation*}
    T(\boldsymbol{\alpha})\ll H(XY)^{1+\epsilon}\left(\frac{1}{q_1}+\frac{1}{X}+\frac{q_1}{HX^k}\right)^{2^{1-k}}\left(\frac{1}{q_2}+\frac{1}{Y}+\frac{q_2}{HX^{l}Y^{k-l}}\right)^{2^{1-k}}
\end{equation*}


%

\end{proof}

\bigskip

In order to describe Lemma 3.2, we introduce the set of smooth numbers $$\mathcal{A}(Y,R)=\{1\leq n\leq Y|\ p\ \textrm{prime}\ \textrm{and}\ p|n\Rightarrow p\leq R\}.$$ 
The following lemma provides an upper bound for the exponential sum in $(\ref{11'})$ with $A=[1,X]\cap \Z$ and $B=\mathcal{A}(Y,R)$ where $R=Y^{\eta}$ with sufficiently small $\eta>0$. 
To describe the following lemma, it is convenient to define an exponential sum $S(\boldsymbol{\alpha})=S(\boldsymbol{\alpha}; H,X,Y,R)$ by 

$$S(\boldsymbol{\alpha})=\dsum_{1\leq h\leq H}\biggl|\dsum_{1\leq x\leq X}\dsum_{y\in \mathcal{A}(Y,R)}e(hg(x,y))\biggr|$$
\begin{lem}
Let $k$ and $l$ be natural numbers satisfying $1\leq l\leq \lfloor k/2 \rfloor$ and $k\geq 3$, and consider real numbers $\alpha_k$ and $\alpha_j\ (0\leq j\leq l)$. Suppose that $X,\ Y$ and $H$ are positive real numbers sufficiently large in terms of $k$ with $Y\leq X$. 
Let $Z$ be a real number such that
$$1\leq Z\leq \min\{(HX^l)^{1/(2(k-l)}, Y\},$$
 and let $N$ be a real number with $Z^{\frac{1}{2}(k-l)}\leq N\leq HX^{l}Y^{k-l}Z^{-\frac{1}{2}(k-l)}$. Suppose that there exist $q\in \N$ and $a\in \Z$ satisfying $ q\leq N, (q,a)=1$ and $|q\alpha_l-a|\leq N^{-1}.$ Then, for each $\epsilon>0,$ one has
\begin{equation}\label{15'}
\begin{aligned}
   S(\boldsymbol{\alpha})\ll \frac{H(XY)^{1+\epsilon}}{(q+HX^lY^{k-l}|q\alpha_l-a|)^{1/(2^{l+1}(k-l))}}+H(XY)^{1+\epsilon}Z^{-1/2^{l+2}}.
\end{aligned}
\end{equation}
\end{lem}

To prove Lemma 3.2, we shall use following two propositions.

\bigskip

\begin{prop}
Let $f$ be an arithmetic function. Suppose that $Y,M_1,\ldots,M_r$ and $R$ satisfy $1\leq R<Y$, $1\leq M_i<Y\ (1\leq i\leq r)$ and $M_1\cdots M_r<Y.$ Then, one has
\begin{equation*}
    \dsum_{y\in \mathcal{A}(Y,R)}f(y)\ll F(Y,M_1,\ldots,M_r,R)+G(Y,M_1,\ldots,M_r,R),
\end{equation*}
where
\begin{align*}
&\qquad\quad F(Y,M_1,\ldots,M_r,R)
=\dsum_{M_1<v_1\leq M_1R}\cdots\dsum_{M_r<v_r\leq M_rR}\dsum_{u\in\mathcal{A}(Y/(M_1\cdots M_r),R)}\left|f(v_1\cdots v_ru)\right|\\ &\textrm{and} \\
&\qquad\quad G(Y,M_1,\ldots,M_r,R)=\dsum_{j=1}^r M_1\cdots M_jR^{j-1}\sup_{1\leq y\leq M_j}|f(y)|.
\end{align*}

\end{prop}

\begin{proof}
Observe that if $y\in \mathcal{A}(Y,R)$ and $y>M_1,$ then there is a divisor $v_1$ of $y$ satisfying $M_1<v_1\leq M_1R$. Moreover, one then has $y/v_1\in \mathcal{A}(Y/M_1,R).$ Hence

\begin{align*}
    \dsum_{y\in \mathcal{A}(Y,R)}f(y)&\ll\dsum_{M_1<v_1\leq M_1R}\dsum_{u\in \mathcal{A}(Y/M_1,R)}|f(v_1u)|+\dsum_{y\leq M_1}|f(y)|\\
    &\ll     \dsum_{M_1<v_1\leq M_1R}\dsum_{u\in \mathcal{A}(Y/M_1,R)}|f(v_1u)|+M_1\sup_{1\leq y\leq M_1}|f(y)|.
\end{align*}
Inductively applying this relation, the conclusion follows.\end{proof}

\bigskip

Proposition 3.4 provides the upper bound for a certain type of exponential sums in many variables. To prove Proposition 3.4, we follow the argument in the proof of  [$\ref{ref10}$, Lemma 4.3]. To describe the following Proposition, it is convenient to define an exponential sum $\Xi(\beta)=\Xi(\beta; \boldsymbol{V}, U, L)$ by
$$\Xi(\beta)=\dsum_{v_1,\ldots,v_r}\dsum_{u_1,u_2}\biggl|\dsum_{1\leq x\leq L}e(\beta(v_1\cdots v_r)^{k-l}(u_1^{k-l}-u_2^{k-l})x) \biggr|,$$
where the summation is over integers satisfying
\begin{equation}\label{lab3.6}
    1\leq v_i\leq V_i\ (1\leq i\leq r)\ \textrm{and}\ U/2< u_1,u_2\leq U.
\end{equation}

\begin{prop}
Let $k$ and $l$ be natural numbers satisfying $1\leq l\leq \lfloor\frac{k}{2} \rfloor$ and $k\geq 3.$ Suppose that $L$ is a positive real number sufficiently large in terms of $k$. Let $U,V_1,\ldots, V_r$ be positive numbers satisfying
\begin{equation}\label{con3.7}
\begin{aligned}
&1\leq U\leq \tfrac{1}{2}(LV_1^{k-l}\cdots V_r^{k-l})^{1/(2(k-l)-1)}\\
  &1\leq V_r<\tfrac{1}{2}L^{1/(k-l)}U^{-\frac{1}{2(k-l)}}\\ 
  &1\leq V_i< \tfrac{1}{2}L^{1/(k-l)}V_{i+1}\cdots V_{r}U^{-1}\  (1\leq i\leq r-1).
\end{aligned}
\end{equation}
 Let $\beta\in\R$, $Q=V_1^{k-l}\cdots V_r^{k-l}U^{k-l}L$ and $N$ be a real number with $U^{\frac{1}{2}(k-l)}\leq N\leq QU^{-\frac{1}{2}(k-l)}$. Suppose that there exist $q\in \N$ and $a\in \Z$ satisfying $ q\leq N, (q,a)=1$ and $|q\beta-a|\leq N^{-1}.$ Then, for each $\epsilon>0,$ one has
\begin{equation}\label{20}
\begin{aligned}
 \Xi(\beta)\ll \frac{V_1\cdots V_rU^2L^{1+\epsilon}}{(q+Q|q\beta-a|)^{1/(k-l)}}+V_1\cdots V_rL^{1+\epsilon}U^{3/2},
\end{aligned}
\end{equation}
where the implicit constant may depend on $r.$

\end{prop}

\begin{proof}
By summing over $x$ in the inner sum of $\Xi(\beta)$, one has

\begin{equation}\label{18'}
   \Xi(\beta) \ll  \dsum_{v_1,\ldots,v_r}\dsum_{u_1,u_2}\frac{L}{1+L\|\beta (v_1\cdots v_r)^{k-l}(u_1^{k-l}-u_2^{k-l})\|},
\end{equation}
where the summations are over $v_1,\ldots,v_r,u_1,u_2$ satisfying $(\ref{lab3.6}).$ Plainly, one may restrict the summation to satisfy
\begin{equation}\label{22}
   U/2\leq u_2<u_1\leq U\ \textrm{and} \ \|\beta(v_1\cdots v_r)^{k-l}(u_1^{k-l}-u_2^{k-l})\|< L^{-1}U^{1/2}.
\end{equation}
Indeed, the contribution on the right hand side of ($\ref{18'}$) of the remaining summands is $O(V_1\cdots V_rLU^{3/2}).$ 
For given $u_1,u_2,v_1,\ldots,v_r$, we may choose $n_0$ so that
$$\|\beta (v_1\cdots v_r)^{k-l}(u_1^{k-l}-u_2^{k-l})\|=|\beta(v_1\cdots v_r)^{k-l}(u_1^{k-l}-u_2^{k-l})-n_0|.$$
Let $\mathcal{R}_{r-1}=(LV_r^{k-l}U^{-1/2})^{1/2}$. By Dirichlet's approximation theorem, there exist $q_{r-1}\in \N$ and $\ a_{r-1}\in \Z$  with  $q_{r-1}\leq \mathcal{R}_{r-1}$ and $(q_{r-1},a_{r-1})=1$ such that 
\begin{equation}\label{22'}
    |\beta(v_1\cdots v_{r-1})^{k-l}(u_1^{k-l}-u_2^{k-l})q_{r-1}-a_{r-1}|<1/\mathcal{R}_{r-1}.
\end{equation}
If $a_{r-1}=0$, then $q_{r-1}=1.$ Thus, on recalling $(\ref{con3.7})$, it follows from ($\ref{22}$) and ($\ref{22'}$) that
\begin{equation*}
    \left|\frac{n_0}{v_r^{k-l}}-\frac{a_{r-1}}{q_{r-1}}\right|v_r^{k-l} q_{r-1}<\frac{\mathcal{R}_{r-1}U^{1/2}}{L}+\frac{V_r^{k-l}}{\mathcal{R}_{r-1}}<1.
\end{equation*}
 Thus, one has $q_{r-1}n_0=v_r^{k-l}a_{r-1}$. If $a_{r-1}=0$, then $q_{r-1}=1.$ Hence, in all cases, $q_{r-1}|v_r^{k-l}.$ Let $q_{r-1}=q_{1,r-1}q_{2,r-1}^2\cdots q_{k-l,r-1}^{k-l}$ where $q_{k-l,r-1}$ is maximal and $q_{1,r-1},\ldots,q_{k-l-1,r-1}$ are squarefree and coprime in pairs. Then, $q_{1,r-1}\cdots q_{k-l,r-1}|v_r$. Thus, by writing $$v_r=q_{0,r-1}q_{1,r-1}\cdots q_{k-l,r-1},$$ the bound ($\ref{18'}$) may be replaced by
\begin{equation}\label{22''}
\begin{aligned}
    &  \dsum_{v_1,\ldots,v_{r-1}}\dsum_{u_1,u_2}\dsum_{q_{0,r-1}\leq V_r/(q_{1,r-1}\cdots q_{k-l,r-1})}\frac{L}{Z_0}+O(V_1\cdots V_r LU^{3/2}),
\end{aligned}
\end{equation}
where $$Z_0=1+L(q_{0,r-1}\cdots q_{k-l,r-1})^{k-l}|\beta (v_1\cdots v_r)^{k-l}(u_1^{k-l}-u_2^{k-l})-a_{r-1}/q_{r-1}|.$$
Since we have $q_{r-1}\leq (q_{1,r-1}q_{2,r-1}\cdots q_{k-l,r-1})^{k-l}$, by applying the upper bound (1) from Appendix A, the inner sum of $(\ref{22''})$ is
\begin{equation}\label{24}
\begin{aligned}
  &\ll  \frac{LV_r}{q_{1,r-1}\cdots q_{k-l,r-1}(1+LV_r^{k-l}|\beta (v_1\cdots v_{r-1})^{k-l}(u_1^{k-l}-u_2^{k-l})-a_{r-1}/q_{r-1}|)^{1/(k-l)}}  \\
  &\ll \frac{LV_r}{(q_{r-1}+LV_r^{k-l}|q_{r-1}\beta (v_1\cdots v_{r-1})^{k-l}(u_1^{k-l}-u_2^{k-l})-a_{r-1}|)^{1/(k-l)}}.
\end{aligned}
\end{equation}
Thus, by substituting the bound ($\ref{24}$) into the inner sum in ($\ref{22''}$), the first term in ($\ref{22''}$) is
\begin{equation}\label{eq3.14}
   \ll \dsum_{v_1,\ldots,v_{r-1}}\dsum_{u_1,u_2}\frac{LV_r}{(q_{r-1}+LV_r^{k-l}|q_{r-1}\beta (v_1\cdots v_{r-1})^{k-l}(u_1^{k-l}-u_2^{k-l})-a_{r-1}|)^{1/(k-l)}}
\end{equation}

We analyse the expression ($\ref{eq3.14}$) by an inductive arguments similar to the argument first described. At the $i$th step of this argument, with $1\leq i\leq r,$ one finds that \begin{equation}\label{eq3.15}
\Xi(\beta) \ll B_i+O(V_1\cdots V_rL^{1+\epsilon}U^{3/2}),
\end{equation} where
\begin{equation}\label{eq3.16}
    B_i=\dsum_{v_1,\ldots, v_{r-i}}\dsum_{u_1,u_2} \frac{LV_rV_{r-1}\cdots V_{r-i+1}(\log L)^i}{(q_{r-i}+LV_r^{k-l}\cdots V_{r-i+1}^{k-l}|q_{r-i}\beta(v_1\cdots v_{r-i})^{k-l}(u_1^{k-l}-u_2^{k-l})-a_{r-i}|)^{1/(k-l)}}
\end{equation}
in which $q_{r-i}\in \Z$ and $a_{r-i}\in \N$ satisfy $(q_{r-i},a_{r-i})=1,$ 
$$q_{r-i}\leq \mathcal{R}_{r-i}=\left\{\begin{array}{l}
    L^{1/2}(V_r\cdots V_{r-i+1})^{(k-l)/2}\ \ \ \ \ \ \ \  \ \ \textrm{when} \ \ i\geq 2\\
    L^{1/2}V_r^{(k-l)/2}U^{-1/4}\ \ \ \ \ \ \ \ \ \ \ \ \ \ \ \  \ \  \textrm{when} \ \ i=1.
    \end{array}
    \right.$$ 
and
$$|\beta(v_1\cdots v_{r-i})^{k-l}(u_1^{k-l}-u_2^{k-l})q_{r-i}-a_{r-i}|<1/\mathcal{R}_{r-i}.$$

The case $i=1$ obviously follows by ($\ref{22'}$), ($\ref{22''}$) and ($\ref{eq3.14}$). Assume that $(\ref{eq3.15})$ holds for a particular $i$ with $i\leq r-1.$ We shall show that ($\ref{eq3.15}$) holds for $i+1.$
Plainly, one may restrict the summation in $(\ref{eq3.16})$ to satisfy
\begin{equation}\label{25}
q_{r-i}+LV_r^{k-l}\cdots V_{r-i+1}^{k-l}|q_{r-i}\beta(v_1\cdots v_{r-i})^{k-l}(u_1^{k-l}-u_2^{k-l})-a_{r-i}|< U^{\frac{1}{2}(k-l)}.
\end{equation}
Indeed, the contribution of the remaining summands is $ O(V_1\cdots V_rLU^{3/2}).$
Let $\mathcal{R}_{r-i-1}=L^{1/2}(V_r\cdots V_{r-i})^{(k-l)/2}.$ By Dirichlet's approximation theorem, there exist $q_{r-i-1}\in \N,a_{r-i-1}\in \Z$ with $q_{r-i-1}\leq \mathcal{R}_{r-i-1}$ and $(q_{r-i-1},a_{r-i-1})=1$ such that 
\begin{equation}\label{26}
    |\beta(v_1\cdots v_{r-i-1})^{k-l}(u_1^{k-l}-u_2^{k-l})q_{r-i-1}-a_{r-i-1}|<1/\mathcal{R}_{r-i-1}.
\end{equation}
Notice that if $a_{r-i-1}=0$ then $q_{r-i-1}=1.$ Thus, on recalling  $(\ref{con3.7})$, it follows from $(\ref{25})$ and $(\ref{26})$ that
\begin{equation*}
    \left|\frac{a_{r-i}}{q_{r-i}v_{r-i}^{k-l}}-\frac{a_{r-i-1}}{q_{r-i-1}}\right|q_{r-i}v_{r-i}^{k-l}q_{r-i-1}<\frac{\mathcal{R}_{r-i-1}U^{\frac{k-l}{2}}}{LV_r^{k-l}\cdots V_{r-i+1}^{k-l}}+\frac{U^{\frac{k-l}{2}}V_{r-i}^{k-l}}{\mathcal{R}_{r-i-1}}<1.
\end{equation*}
 Thus, one has $q_{r-i}v_{r-i}^{k-l}a_{r-i-1}=q_{r-i-1}a_{r-i}.$ If $a_{r-i}=0$ then $q_{r-i}=1.$ Hence, in all cases, $q_{r-i}|q_{r-i-1}.$ By writing $q'_{r-i-1}=q_{r-i-1}/q_{r-i},$ one has $$v_{r-i}^{k-l}a_{r-i-1}=q'_{r-i-1}a_{r-i}.$$
Since $(q'_{r-i-1},a_{r-i-1})=1,$ one has $q_{r-i-1}'|v_{r-i}^{k-l}.$ Similarly, one has $$q_{r-i-1}'=q_{1,r-i-1}q_{2,r-i-1}^{2}\cdots q_{k-l,r-i-1}^{k-l},$$ where $q_{k-l,r-i-1}$ is maximal and $q_{1,r-i-1},\ldots,q_{k-l-1,r-i-1}$ are squarefree and coprime in pairs. Then, $q_{1,r-i-1}q_{2,r-i-1}\cdots q_{k-l,r-i-1}|v_{r-i}$. Thus, by writing $v_{r-i}=q_{0,r-i-1}\tau$ with $$\tau=q_{1,r-i-1}\cdots q_{k-l,r-i-1},$$ the equation ($\ref{eq3.16}$) may be replaced by 
\begin{equation}\label{28}
    \dsum_{v_1,\ldots,v_{r-i-1}}\dsum_{u_1,u_2}
    \dsum_{q_{0,r-i-1}<V_{r-i}/\tau}\frac{LV_r\cdots V_{r-i+1}(\log L)^i}{Z}+O(V_1\cdots V_r L^{1+\epsilon}U^{3/2}),
\end{equation}
where $$Z^{k-l}=q_{r-i}\biggl(1+L(V_r\cdots V_{r-i+1})^{k-l}\bigl(q_{0,r-i-1}\tau)^{k-l}\bigl|\beta(v_1\cdots v_{r-i-1})^{k-l}(u_1^{k-l}-u_2^{k-l})-\frac{a_{r-i-1}}{q_{r-i-1}}\bigr|\biggr).$$  By applying the upper bound (2) from Appendix A, the inner sum in ($\ref{28}$) is 
\begin{equation*}
\begin{aligned}
& \ll\frac{LV_r\cdots V_{r-i}(\log L)^{i+1}}{q_{r-i}^{1/(k-l)}\tau(1+L(V_r\cdots V_{r-i})^{k-l}|\beta(v_1\cdots v_{r-i-1})^{k-l}(u_1^{k-l}-u_2^{k-l})-a_{r-i-1}/q_{r-i-1}|)^{1/(k-l)}}.
\end{aligned}
\end{equation*}
Since we have $q'_{r-i-1}\leq \tau^{k-l}$ and $q_{r-i-1}=q'_{r-i-1}q_{r-i}$, we find that this expression is
\begin{equation}\label{3.18}
\begin{aligned}
&\leq \frac{LV_r\cdots V_{r-i}(\log L)^{i+1}}{q_{r-i-1}^{1/(k-l)}(1+L(V_r\cdots V_{r-i})^{k-l}|\beta(v_1\cdots v_{r-i-1})^{k-l}(u_1^{k-l}-u_2^{k-l})-a_{r-i-1}/q_{r-i-1}|)^{1/(k-l)}}\\
&= \frac{LV_r\cdots V_{r-i}(\log L)^{i+1}}{(q_{r-i-1}+L(V_r\cdots V_{r-i})^{k-l}|q_{r-i-1}\beta(v_1\cdots v_{r-i-1})^{k-l}(u_1^{k-l}-u_2^{k-l})-a_{r-i-1}|)^{1/(k-l)}}.
\end{aligned}
\end{equation}
Hence, by substituting the bound ($\ref{3.18}$) into the inner sum ($\ref{28}$), the first term in ($\ref{28}$) is seen to be $O(B_{i+1}).$
Thus, this confirms that the bound ($\ref{eq3.15}$) holds for $i+1$. 
Therefore, one infers by induction that
\begin{equation}\label{30}
    \Xi(\beta)\ll \Xi(\beta)+O(V_1\cdots V_r L^{1+\epsilon}U^{3/2}),
\end{equation}
where $$\Xi_1(\beta)=\dsum_{u_1,u_2}\frac{LV_r\cdots V_{1}(\log L)^r}{(q_0+LV_r^{k-l}\cdots V_1^{k-l}|q_0\beta(u_1^{k-l}-u_2^{k-l})-a_0|)^{1/(k-l)}}$$
in which ($q_0,a_0$)=1, $q_0\leq \mathcal{R}_0=L^{1/2}(V_r\cdots V_1)^{(k-l)/2}$ and $|q_0\beta(u_1^{k-l}-u_2^{k-l})-a_0|<1/\mathcal{R}_0.$
By Cauchy's inequality, one has
\begin{equation}\label{31}
    \Xi_1(\beta) \ll UL^{1+\epsilon/2}V_r\cdots V_1 \Xi_{2}(\beta)^{1/2},
\end{equation}
where $$\Xi_2(\beta)=\dsum_{u_1,u_2}\frac{1}{(q_0+LV_r^{k-l}\cdots V_1^{k-l}|q_0\beta(u_1^{k-l}-u_2^{k-l})-a_0|)^{2/(k-l)}}.$$

Plainly, one may also restrict the summation to satisfy
\begin{equation}\label{36''}
    q_0+LV_r^{k-l}\cdots V_1^{k-l}|q_0\beta(u_1^{k-l}-u_2^{k-l})-a_0|<U^{\frac{1}{2}(k-l)}.
\end{equation}
Indeed, the contribution from the summation in ($\ref{31}$) arising from remaining terms is $O(V_1\cdots V_r L^{1+\epsilon}U^{3/2}).$
We put 
$$ j=(u_1,u_2),\ \ \ n=u_2/j,\ \ \ m=(u_1-u_2)/j$$
so that 
$$j\leq U,\ \ m\leq U/j,\ \ U/(2j)<n < n+m\leq U/j,\ \ (n,n+m)=1. $$
Now, $q_0$ and $a_0$ will depend on $n,m,j.$ Let $S=((U/j)^{k-l-1}LV_r^{k-l}\cdots V_1^{k-l})^{1/2}$. Then for fixed $j$ and $m$, by Dirichlet's approximation theorem, there exists $c\in \Z$ and $s\in \N$ such that 
$$(c,s)=1,\ \ \ s\leq S,\ \ \ \textrm{and}\ \ \ |\beta j^{k-l}m-c/s|\leq (sS)^{-1}.$$
Notice again that if $c=0,$ then $s=1$. Let $D=((n+m)^{k-l}-n^{k-l})/m.$ Then 
$$D=\frac{k-l}{m}\dint_{n}^{n+m}x^{k-l-1}dx,$$
and so 
$$ (k-l)(U/(2j))^{k-l-1}\leq D\leq (k-l)(U/j)^{k-l-1}.$$
Thus, from ($\ref{36''}$), one has $q_0\leq U^{\frac{1}{2}(k-l)}$, and
$$\left|\beta j^{k-l}m-\frac{a_0}{q_0D}\right|\leq \frac{U^{\frac{1}{2}(k-l)}}{q_0DLV_r^{k-l}\cdots V_1^{k-l}} $$
Therefore, since  $U\leq \tfrac{1}{2}(LV_1^{k-l}\cdots V_r^{k-l})^{1/(2(k-l)-1)},$ one has
$$|q_0Dc-sa_0|=\left|\frac{c}{s}-\frac{a_0}{q_0D}\right|sq_0D\leq \frac{U^{\frac{1}{2}(k-l)}D}{S}+\frac{U^{\frac{1}{2}(k-l)}S}{LV_r^{k-l}\cdots V_1^{k-l}}<1.$$
Thus, we see that $cq_0D=a_0s.$ Hence $q_0|s.$ Let $s_1=s/q_0.$ Then $cD=a_1s_1.$ Hence $c|a_1$ and $s_1|D.$ Therefore, as ($n,\ n+m$)=1, we have $(n(n+m),s_1)=1.$ Thus, one finds that 
\begin{equation}\label{37}
   \Xi_2(\beta) \ll \dsum_{j\leq U}\dsum_{m\leq U/j}\dsum_{s_1|s}\Xi_3(\beta)
\end{equation}
where
$$\Xi_3(\beta)=\dsum_{n}\frac{(s_1/s)^{2/(k-l)}}{(1+LV_r^{k-l}\cdots V_1^{k-l}(U/(2j))^{k-l-1}|\beta j^{k-l}m-c/s|)^{2/(k-l)}}$$
in which the innermost sum is over $n$ satisfying
\begin{equation}\label{38}
    n\leq U/j,\ \ \ (n(n+m),s_1)=1,\ \ \ s_1|D.
\end{equation}
The third condition $s_1|D$ implies that 
\begin{equation}\label{eq3.26}
    (n+m)^{k-l}\equiv n^{k-l}\ (\textrm{mod}\ s_1m).
\end{equation}
Since we have $(n,n+m)=1$ and $(n,s_1)=1$, one finds that $(n,s_1m)=1$. Thus, there exists $n_0$ such that $nn_0\equiv 1\ (\textrm{mod}\ s_1m).$ Hence, the congruence $(\ref{eq3.26})$ is equivalent to 
$$(1+mn_0)^{k-l}\equiv 1\ (\textrm{mod}\ s_1m).$$
Notice that the congruence $y^{k-l}\equiv 1\ (\textrm{mod}\ s_1m)$ has $v$ solutions modulo $s_1m$, say $g_1,\ldots,g_v$, where $v\ll (s_1m)^{\epsilon}.$ Hence $1+mn_0\equiv g_i\ (\textrm{mod}\ s_1m)$ for some $1\leq i\leq v$. Thus, there are at most $v$ choices for $n_0$, and so for $n,$ modulo $s_1.$
Then, on noting that $s_1m$ is bounded above by powers of $L$ which may depend on $r$, we see that
\begin{equation}\label{39'}
   \Xi_3(\beta) \ll \left(\frac{U}{js_1}+1\right)\frac{L^{\epsilon/4}(s_1/s)^{2/(k-l)}}{(1+LV_r^{k-l}\cdots V_1^{k-l}(U/(2j))^{k-l-1}|\beta j^{k-l}m-c/s|)^{2/(k-l)}},
\end{equation}
where the implicit may depend on $r.$
By examining separately the contribution arising from the terms $U/(js_1)$ and $1$ in the first factor of  $(\ref{39'})$, we see from ($\ref{37}$) that
 $$\Xi_2(\beta) \ll A+O(L^{\epsilon/2}U),$$ where 
$$A=\dsum_{j\leq U}\dsum_{m\leq U/j}\frac{L^{\epsilon/4}Uj^{-1}}{(s+LV_r^{k-l}\cdots V_1^{k-l}(U/(2j))^{k-l-1}|\beta j^{k-l}ms-c|)^{2/(k-l)}}.$$
Plainly, we may restrict the inner sum in $A$ to those $m$ satisfying 
\begin{equation*}
    s+LV_r^{k-l}\cdots V_1^{k-l}(U/(2j))^{k-l-1}|\beta j^{k-l}ms-c|< U^{\frac{1}{2}(k-l)}.
\end{equation*}
Indeed, the contribution on $A$ of the remaining sum is $O(L^{\epsilon}U)$, and thus the right hand side in ($\ref{31}$) is $O(V_1\cdots V_r L^{1+\epsilon}U^{3/2}).$
Let $T=(LV_r^{k-l}\cdots V_1^{k-l}(U/j)^{k-l})^{1/2}$. Then for fixed $j$, by Dirichlet's approximation theorem there exists $d$ and $t$ with 
\begin{equation*}
    (d,t)=1,\ \ \ t\leq T\ \ \ \textrm{and}\ \ \ |\beta j^{k-l}-d/t|\leq (tT)^{-1}.
\end{equation*}
Notice again that if $d=0$, then $t=1$. Then, on recalling $U\leq \tfrac{1}{2}(LV_1^{k-l}\cdots V_r^{k-l})^{1/(2(k-l)-1)}$, for fixed $j$ and $m$ one has
\begin{equation*}
    \left|\frac{c}{ms}-\frac{d}{t}\right|tms\leq \frac{U^{\frac{1}{2}(k-l)}U/j}{T}+\frac{TU^{\frac{1}{2}(k-l)}}{LV_r^{k-l}\cdots V_1^{k-l}(U/(2j))^{k-l-1}}<1.
\end{equation*}
Thus, we see that $ct=dms,$ and so $s|t$. Let $t_1=t/s$. Then $ct_1=dm.$ Thus $t_1|m.$ Let $m_1=m/t_1.$ Therefore, the summation $A$ is bounded above by
\begin{equation*}
   \dsum_{j\leq U}\dsum_{t_1|t}(t_1/t)^{2/(k-l)}A_1,
\end{equation*}
where 
$$A_1=\dsum_{m_1\leq U/(jt_1)}\frac{L^{\epsilon/4}Uj^{-1}}{(1+LV_r^{k-l}\cdots V_1^{k-l}(U/(2j))^{k-l-1}m_1t_1|\beta j^{k-l}-d/t|)^{2/(k-l)}}.$$
Hence, by applying the upper bound (3) from Appendix A, the innermost sum is 
\begin{equation*}
   A_1 \ll \frac{L^{\epsilon/4}U^2j^{-2}t_1^{-1}\log X}{(1+LV_r^{k-l}\cdots V_1^{k-l}(U/j)^{k-l}|\beta j^{k-l}-d/t|)^{2/(k-l)}}
\end{equation*}
Hence, one finds that
\begin{equation}\label{42'}
   \Xi_2(\beta) \ll L^{\epsilon/2}\left(U+\dsum_{j\leq U}\frac{U^2j^{-2}}{(t+LV_r^{k-l}\cdots V_1^{k-l}(U/j)^{k-l}|\beta j^{k-l}t-d|)^{2/(k-l)}}\right),
\end{equation}
and one may restrict the summation over $j$ to satisfy
\begin{equation}\label{42}
    t+LV_r^{k-l}\cdots V_1^{k-l}(U/j)^{k-l}|\beta j^{k-l}t-d|<\frac{1}{2}(U/j)^{(k-l)/2}.
\end{equation}
Let $\overline{\mathcal{R}}=Q^{1/2}$. Then by Dirichlet's approximation theorem, there exists $\overline{q}$ and $\overline{a}$ such that 
\begin{equation*}
    (\overline{q},\overline{a})=1,\ \ \ \overline{q}\leq \overline{\mathcal{R}}\ \ \ \textrm{and}\ \ \ |\overline{q}\beta-\overline{a}|\leq \overline{\mathcal{R}}^{-1}.
\end{equation*}
Since $t$ is non-zero, one finds that $j\leq U.$ Therefore, on recalling $$U\leq \tfrac{1}{2}(LV_1^{k-l}\cdots V_r^{k-l})^{1/(2(k-l)-1)},$$ when $j$ satisfies $(\ref{42})$, one has
\begin{equation*}
    \left|\frac{\overline{a}}{\overline{q}}-\frac{d}{j^{k-l}t}\right|j^{k-l}t\overline{q}< \frac{(U/j)^{(k-l)/2}j^{k-l}}{2\overline{\mathcal{R}}}+\frac{\overline{\mathcal{R}}(U/j)^{(k-l)/2}}{2LV_r^{k-l}\cdots V_1^{k-l}(U/j)^{k-l}}\leq 1.
\end{equation*}
Thus, we see that $j^{k-l}t\overline{a}=\overline{q}d$. Hence $t|\overline{q}.$ Let $\overline{q}'=\overline{q}/t.$ Then $j^{k-l}\overline{a}=\overline{q}'d.$ Hence $\overline{q}'|j^{k-l}$. Let $\overline{q}'=q_{1}'q_{2}'^2\cdots q_{k-l}'^{k-l}$ where $q_{k-l}'$ is maximal and $q_{1}',\ldots,q_{k-l-1}'$ are square-free and coprime in pairs. Then $q_{1}'\cdots q_{k-l}'|j$, so the summation over $j$ in ($\ref{42'}$) is
\begin{equation}\label{44}
 \ll \dsum_{w}\dsum_{t,q_{1}',\ldots,q_{k-l}'}\frac{U^2(wq_{1}'\cdots q_{k-l}')^{-2}t^{-2/(k-l)}}{(1+LV_r^{k-l}\cdots V_1^{k-l}U^{k-l}|\beta-\overline{a}/\overline{q}|)^{2/(k-l)}},
\end{equation}
where the inner sum is over $t,q_{1}',\ldots,q_{k-l}'$ satisfying $\overline{q}=tq_{1}'q_{2}'^2\cdots q_{k-l}'^{k-l}$. Thus, on recalling $Q=(V_1\cdots V_r U)^{k-l}L$, since we have $\overline{q}^{2/(k-l)}\leq  t^{2/(k-l)}q_1'q_2'^2\cdots q_{k-l}'^{k-l}$, the bound ($\ref{44}$) is
\begin{equation*}
    \ll \frac{U^2L^{\epsilon/2}}{(\overline{q}+Q|\beta \overline{q}-\overline{a}|)^{2/(k-l)}}.
\end{equation*}

When $\overline{q}+Q|\overline{q}\beta-\overline{a}|\geq \frac{1}{2}U^{(k-l)/2}$ we are done, so we may suppose that $\overline{q}+Q|\overline{q}\beta-\overline{a}|\leq \frac{1}{2}U^{(k-l)/2}.$ Thus, 
$$\left|\frac{\overline{a}}{\overline{q}}-\frac{a}{q}\right|q\overline{q}<\frac{U^{(k-l)/2}N}{2Q}+\frac{U^{(k-l)/2}}{2N}\leq 1,$$
since by assumption, $U^{\frac{1}{2}(k-l)}\leq N \leq QU^{-\frac{1}{2}(k-l)}.$ Hence, one has $q=\overline{q}, a=\overline{a}.$ Therefore, we complete the proof.
\end{proof}

We now turn to prove Lemma 3.2. To prove this lemma, we exploit Proposition 3.3 and Proposition 3.4.
\begin{proof}[Proof of Lemma 3.2]
 Note that 
\begin{equation}\label{eq3.311}
S(\boldsymbol{\alpha})\leq \dsum_{1\leq h\leq H}\dsum_{y\in \mathcal{A}(Y,R)}\biggl|\dsum_{1\leq x\leq X}e(hg(x,y))\biggr|.
\end{equation}
We shall first apply Proposition 3.3 to the right hand side in ($\ref{eq3.311}$). Let $M_1,\ldots,M_r$ be real numbers with $M_1,\ldots, M_r\geq 1$ satisfying $M_1\cdots M_r=Y/Z$. We will define each of $M_1,\ldots,M_r$ later for applications of Proposition 3.4. For now, we temporarily assume that such $M_1,\ldots,M_r$ exist. Then, by applying Proposition 3.3 with $R=Y^{\eta}$ and 
$$f(y)=\dsum_{1\leq x\leq X}e(hg(x,y)),$$
 one finds from ($\ref{eq3.311}$) that
\begin{equation}\label{35'}
  S(\boldsymbol{\alpha})  \ll S_1(\boldsymbol{\alpha})+\dsum_{1\leq h\leq H}\dsum_{j=1}^r M_1\cdots M_jR^{j-1}\sup_{1\leq y\leq M_j}|f(y)|,
\end{equation}
where 
$$S_1(\boldsymbol{\alpha})=\dsum_{1\leq h\leq H}\dsum_{M_1<v_1\leq M_1R}\cdots\dsum_{M_r<v_r\leq M_rR}\\
    \dsum_{u\in\mathcal{A}(Z,R)}\left|\dsum_{1\leq x\leq X}e\left(hg(x,v_1\ldots v_ru)\right)\right|.$$
Since by $Y/(M_1\cdots M_r)=Z$ and $\sup_{1\leq y\leq M_i}|f(y)|\leq X$ for all $h$, it follows trivially that
\begin{equation*}
    \dsum_{1\leq h\leq H}\dsum_{j=1}^r M_1\cdots M_jR^{j-1}\sup_{1\leq y\leq M_i}|f(y)|\ll H(XY)^{1+\epsilon}Z^{-1/2^{l+2}}.
\end{equation*}

 Then, it suffices to bound $S_1(\boldsymbol{\alpha})$. By applying H\"older's inequality, we see that
\begin{equation}\label{eq3.31}
S_1(\boldsymbol{\alpha})\leq (HYR^r)^{1-2^{-l}}S_2(\boldsymbol{\alpha})^{2^{-l}},
\end{equation}
where 
$$S_2(\boldsymbol{\alpha})=\dsum_{1\leq h\leq H}\dsum_{u\in\mathcal{A}(Z,R)}\dsum_{M_1<v_1\leq M_1R}\cdots\dsum_{M_r<v_r\leq M_rR}\biggl|\dsum_{1\leq x\leq X}e\left(hg(x,v_1\cdots v_ru)\right)\biggr|^{2^{l}}.$$
By applying Weyl differencing, we have 
\begin{align}\label{eq3.32}
&S_2(\boldsymbol{\alpha})\ll X^{2^{l}-l-1}\dsum_{1\leq h\leq H}\dsum_{u\in\mathcal{A}(Z,R)}\dsum_{M_1<v_1\leq M_1R}\cdots\dsum_{M_r<v_r\leq M_rR}S_3(\boldsymbol{\alpha}),    
\end{align}
where 
$$S_3(\boldsymbol{\alpha})=\dsum_{|h_1|\leq X}\cdots\dsum_{|h_{l}|\leq X}\dsum_{x\in I_l(\mathbf{h})}e\left(h\Delta_{l}^x(g(x,v_1\cdots v_ru),\mathbf{h})\right)$$
in which $I_l(\mathbf{h})$ is an interval of integers contained in $[-X,X]$. Note here that $$\Delta_{l}^x(g(x,v_1\cdots v_ru),\mathbf{h})=\Delta_l^x(\alpha_k x^k)+\Delta_l^x(\alpha_lx^l(v_1\cdots v_r u)^{k-l})$$ and $\Delta_l^x(\alpha_lx^l(v_1\cdots v_r u)^{k-l})$ does not depend on $x.$
Thus, by splitting $\mathcal{A}(Z,R)$ into dyadic intervals $[Z_0/2^i,Z_0/2^{i-1}]$ with $Z_0\leq Z$, for any $\epsilon>0$, one infers from ($\ref{eq3.32}$) by applying a standard divisor estimate that
\begin{equation*}
    S_2(\boldsymbol{\alpha})\ll X^{2^l-l+\epsilon}\dsum_{M_1<v_1\leq M_1R}\cdots\dsum_{M_r<v_r\leq M_rR}\dsum_{1\leq n\leq l!HX^{l}}\biggl|\dsum_{u}e\left(\alpha_l(v_1\cdots v_ru)^{k-l}n\right)\biggr|
\end{equation*}
where the innermost sum is over $u\in  \mathcal{A}(Z,R)$ and $Z_0/2\leq u\leq Z_0$ for some $Z_0$ with $Z_0\leq Z$. Since we have $M_1\cdots M_r=Y/Z$, by applying Cauchy's inequality, we deduce that
\begin{equation}\label{eq3.34}
    S_2(\boldsymbol{\alpha})\ll X^{2^l-l+\epsilon}(HX^l(Y/Z)R^r)^{1/2}S_4(\boldsymbol{\alpha})^{1/2},
\end{equation}
where $$S_4(\boldsymbol{\alpha})=\dsum_{M_1<v_1<M_1R}\cdots\dsum_{M_r<v_r<M_rR}\dsum_{1\leq n\leq l!HX^{l}}\left|\dsum_{u}e\left(\alpha_l(v_1\cdots v_ru)^{k-l}n\right)\right|^2.$$
By squaring out and change the order of summations, we see that
\begin{equation}\label{eq3.35}
    S_4(\boldsymbol{\alpha})\leq \dsum_{M_1<v_1\leq M_1R}\cdots\dsum_{M_r<v_r\leq M_rR}\ \dsum_{Z_0/2\leq u_1,u_2\leq Z_0}\left|\dsum_{1\leq n\leq l!HX^{l}}e\left(\alpha_l n(v_1\cdots v_r)^{k-l}(u_1^{k-l}-u_2^{k-l})\right)\right|.
\end{equation}


We define $M_1,\ldots, M_r$  here so that we apply Proposition 3.4. We set $r=1/\eta$, and define 
$$a=\biggl\lfloor\frac{1}{\eta}\biggl(1-\frac{\log Z}{\log Y}\biggr)\biggr\rfloor,$$
which satisfies $Y^{a\eta}\leq Y/Z.$ Then, define $M_r,\ldots, M_1$ by
$$M_i=\left\{\begin{array}{l}
    Y^{\eta}\ \ \ \ \ \ \ \  \ \ \textrm{when} \ \ i> r-a\\
    Y/(ZY^{a\eta})\ \textrm{when} \ \ i=r-a\\
    1\ \ \ \ \ \ \ \  \  \ \  \ \ \textrm{when}\ \ 1\leq i<r-a.
    \end{array}
    \right.$$
    Since $Y/Z=M_1\cdots M_r$, this choice for $M_1,\ldots,M_r$ is in accordance with the hypotheses of Proposition 3.3.

 On recalling that $Z_0\leq Z\leq (HX^l)^{1/(2(k-l))}$ and by observing $Y^{\eta}\geq M_r\geq M_{r-1}\geq\cdots\geq M_1,$ one finds that
\begin{align}
    &1\leq M_rR<\frac{1}{2}(HX^l)^{1/(k-l)}Z_0^{-\frac{1}{2(k-l)}},\\ &1\leq M_iR<\frac{1}{2}(HX^l)^{1/(k-l)}(M_{i+1}R)\cdots (M_rR)Z_0^{-1}\ (i=1,\ldots,r-1).
\end{align}
Furthermore, on recalling that $Z\leq (HX^l)^{1/(2(k-l))}$, one deduces that
\begin{equation}
    Z_0\leq \tfrac{1}{2}(HX^l(M_1R)^{k-l}\cdots (M_rR)^{k-l})^{1/(2(k-l)-1)}.
\end{equation}
Thus, on setting $U=Z_0$ and $V_i=M_iR\ (1\leq i\leq r),$ we see that the hypothese of Proposition 3.4 are satisfied.
Then, since we have $Y/Z=M_1\cdots M_r$, by applying Proposition 3.4 to ($\ref{eq3.35}$) with $L=l!HX^l$ and $V_i=M_iR\ (i=1,\ldots, r)$, we obtain that
\begin{equation}\label{eq3.39}
    S_4(\boldsymbol{\alpha})\ll HX^l(Y/Z)R^r Z_0^2X^{\epsilon}\biggl(\frac{1}{(q+(Z_0/Z)^{k-l}HX^lY^{k-l}|q\alpha_l-a|)^{1/(k-l)}}+Z_0^{-1/2}\biggr).
\end{equation}
Therefore, on substituting ($\ref{eq3.39}$) into ($\ref{eq3.34}$) and that into ($\ref{eq3.31}$), we deduce that 
$$S_1(\boldsymbol{\alpha})\ll  X^{\epsilon}HXYR^r(Z_0Z^{-1})^{2^{-l}}\left(\frac{1}{(q+(Z_0/Z)^{k-l}HX^lY^{k-l}|q\beta-a|)^{1/(k-l)}}+Z_0^{-1/2}\right)^{2^{-l-1}}.$$
Since we have $Z_0\leq Z,$ by choosing $\eta$ small in terms of $\epsilon,$ one concludes
$$S_1(\boldsymbol{\alpha})\ll \frac{H(XY)^{1+\epsilon}}{(q+HX^lY^{k-l}|q\beta-a|)^{1/(2^{l+1}(k-l))}}+H(XY)^{1+\epsilon}Z^{-2^{-l-2}}.$$
 Therefore, we are done.

\end{proof}

\bigskip

\section{Proof of Theorem 1.1}

Our goal in this section is to prove Theorem 1.1. As we mentioned in section 2, we reduce the problem of bounding the fractional parts of polynomial in Theorem 1.1 to the corresponding diagonal problems. Thus, we begin this section by examining the corresponding diagonal problem. 

\begin{prop}
Let $\alpha$, $\beta \in \R$ and  $k\in \N$ with $k\geq 2$. Then, for any $\epsilon>0$,  there exists a real number $X(k,\epsilon)$ such that whenever $X\geq X(k,\epsilon)$ and $Y\geq X(k,\epsilon)$ one has
\begin{equation*}
    \min_{\substack{0\leq x\leq X\\0\leq y\leq Y\\(x,y)\neq(0,0)}}\|\alpha x^k+\beta y^k\|\leq X^{-\sigma+\epsilon}Y^{-\sigma+\epsilon},
\end{equation*}
where $\sigma=2^{1-k}.$
\end{prop}
\begin{proof}
Let $H=X^{\sigma-\epsilon}Y^{\sigma-\epsilon}.$ Suppose that there exists no $x$ and $y$ satisfying $$\|\alpha x^k+\beta y^k\|\leq 1/H.$$
By Dirichlet approximation theorem, there exits $q_1,q_2\in \N$ and $a_1,a_2\in \Z$ such that $(q_1,a_1)=(q_2,a_2)=1$ and 
\begin{equation*}
    \left|\alpha-\frac{a_1}{q_1}\right|<\frac{1}{q_1X^{k-1}H},\ \left|\beta-\frac{a_2}{q_2}\right|<\frac{1}{q_2Y^{k-1}H}.
\end{equation*}
If $q_1\leq X$, then by observing that $|q_1\alpha- a_1|<(HX^{k-1})^{-1},$ one has
\begin{equation*}
    \min_{\substack{0\leq x\leq X\\0\leq y\leq Y\\ (x,y)\neq (0,0)}}\|\alpha x^k+\beta y^k\|\leq \|\alpha q_1^k\|\leq q_1^{k-1}\|\alpha q_1\|< X^{k-1}\frac{1}{X^{k-1}H}\leq \frac{1}{H}.
\end{equation*}
Similarly, if $q_2\leq Y,$ one has
\begin{equation*}
    \min_{\substack{1\leq x\leq X\\1\leq y\leq Y\\ (x,y)\neq (0,0)}}\|\alpha x^k+\beta y^k\|\leq\frac{1}{H}.
\end{equation*}
These are contradiction to our assumption. Thus, we may assume that $q_1\geq X$ and $q_2\geq Y$.

 It follows from our assumption by Lemma 2.1 that
\begin{equation}\label{35}
    XY\ll\dsum_{1\leq h\leq H}\bigl|\dsum_{\substack{1\leq x\leq X\\1\leq y\leq Y}}e(h(\alpha x^k+\beta y^k))\bigr|.
\end{equation}
On writing
\begin{equation*}
    \dsum_{\substack{1\leq x\leq X\\1\leq y\leq Y}}e(h(\alpha x^k+\beta y^k))=\dsum_{1\leq x\leq X}e(h\alpha x^k)\dsum_{1\leq y\leq Y}e(h\beta y^k),
\end{equation*}
by Cauchy's inequality, the right hand side in ($\ref{35}$) is
$$\ll \left(\dsum_{1\leq h\leq H}\biggl|\dsum_{1\leq x\leq X}e(h\alpha x^k)\biggr|^2\right)^{1/2}\left(\dsum_{1\leq h\leq H}\biggl|\dsum_{1\leq y\leq Y}e(h\beta y^k)\biggr|^2\right)^{1/2}$$
Then, by applying Weyl differencing, the last expression is for every $\delta>0$
\begin{equation}\label{36}
    \ll X^{1+\delta}Y^{1+\delta}H\left(\frac{1}{q_1}+\frac{1}{X}+\frac{q_1}{X^kH}\right)^{2^{1-k}}\left(\frac{1}{q_2}+\frac{1}{Y}+\frac{q_2}{Y^kH}\right)^{2^{1-k}}.
\end{equation}
Therefore, on recalling that $q_1\geq X$ and $q_2\geq Y,$ one finds that ($\ref{36}$) is
\begin{equation*}
    \ll X^{1+\delta}Y^{1+\delta}X^{-\epsilon}Y^{-\epsilon}.
\end{equation*}
By taking $\delta<\epsilon$, it contradicts $(\ref{35})$. Therefore, we are forced to have 
\begin{equation*}
    \min_{\substack{0\leq x\leq X\\0\leq y\leq Y\\(x,y)\neq(0,0)}}\|\alpha x^k+\beta y^k\|\leq X^{-\sigma+\epsilon}Y^{-\sigma+\epsilon}.
\end{equation*}
\end{proof}

\bigskip

As sketched in section 2, we shall reduce Theorem 1.1 to this diagonal problem through Lemma 3.1 and inductive arguments.

\begin{proof}[Proof of Theorem 1.1]
Let $H=X^{\sigma-\epsilon}$ with $\sigma=\frac{l+2}{l+1}2^{1-k}.$ There exists $q_k$ and $a_k$ with $(q_k,a_k)=1$ such that $q_k\leq X^{k-1}H$ and 
\begin{equation}\label{eq4.4}
    |\alpha_k-a_k/q_k|\leq \frac{1}{q_kX^{k-1}H}.
\end{equation}
Suppose that 
\begin{equation}\label{4.6}
    \frac{1}{H}< \|\alpha_kx^k+\alpha_lx^ly^{k-l}+\alpha_{l-1}x^{l-1}y^{k-l+1}+\cdots +\alpha_0 y^k\| 
\end{equation}
for all $0\leq x,y\leq X$ but $(x,y)= (0,0).$
Then, we may assume $q_k> X$. Indeed, if one were to have $q_k\leq X$, by ($\ref{eq4.4}$), one has
\begin{align*}
    &\min_{\substack{1\leq x,y\leq X\\ (x,y)\neq (0,0)}}\|\alpha_kx^k+\alpha_lx^ly^{k-l}+\alpha_{l-1}x^{l-1}y^{k-l+1}+\cdots+\alpha_0y^k\|\\
    &\leq \min_{1\leq x\leq X}\|\alpha_k x^k\|\leq q_k^{k-1}\|\alpha_kq_k\|\leq \frac{1}{H},
\end{align*}
which contradicts our assumption. 

From our assumption ($\ref{4.6}$), we derive analogues of ($\ref{4.6}$) by an inductive arguments. Specifically, at the $i$-th step, with $0\leq i\leq l$, we shall prove that there exist coefficients $\alpha_{l-i}^{(i)},\ldots,\alpha_{0}^{(i)}$, suitable integral multiples of $\alpha_{l-i},\ldots,\alpha_{0}$, respectively, such that
\begin{equation}\label{4.77}
    \frac{1}{H}\ll \|\alpha_k x^k+\alpha_{l-i}^{(i)}x^{l-i}y^{k-l+i}+\cdots+ \alpha_0^{(i)} y^k\|
\end{equation}
for all $0\leq x\leq X$ and $0\leq y\leq Y_i=X^{1-i/(l+1)}$ but $(x,y)=(0,0).$ 
 The case $i=0$ obviously follows from $(\ref{4.6}).$ Assume that $(\ref{4.77})$ holds for a particular $i\leq l-1.$ We shall show that ($\ref{4.77}$) holds for $i$ replaced by $i+1.$

By Lemma 2.1, the inequality ($\ref{4.77}$) implies that
\begin{equation}\label{4.7}
    XY_i\ll \dsum_{1\leq h\leq H}\biggl|\dsum_{1\leq x\leq X}\dsum_{1\leq y\leq Y_i}e(h(\alpha_k x^k+\alpha_{l-i}^{(i)} x^{l-i}y^{k-l+i}+\cdots+\alpha_0^{(i)} y^k))\biggr|.
\end{equation}
Note that for any $\delta>0$, it follows from Dirichlet's approximation theorem that there exists $q\in\N$ such that
\begin{equation}\label{4.5}
\begin{aligned}
   &q\leq X^{l-i+1}Y_i^{k-l+i}H^{-2^{k-1}+1}X^{-\delta}\\
   &\|q\alpha_{l-i}^{(i)}\|\leq \frac{X^{\delta}}{X^{l-i+1}Y_i^{k-l+i}H^{-2^{k-1}+1}}.
   \end{aligned}
\end{equation}
Thus, on recalling $X<q_k<X^{k-1}H$, it follows by Lemma 3.1 that the right hand side in ($\ref{4.7}$) is
\begin{align*}
    \ll &HX^{1+\epsilon}Y_i\biggl(\frac{1}{q_k}+\frac{1}{X}+\frac{q_k}{X^kH}\biggr)^{2^{1-k}}\biggl(\frac{1}{q}+\frac{1}{Y_i}+\frac{q}{X^{l-i}Y_i^{k-l+i}H}\biggr)^{2^{1-k}}\\
    \ll & HX^{1+\epsilon}Y_iX^{-2^{1-k}}(q^{-2^{1-k}}+Y_i^{-2^{1-k}}+X^{2^{1-k}}H^{-1-\delta}).
\end{align*}
Note that $(XY_i)^{2^{1-k}}\geq H.$
Thus, combining this and the lower bound of ($\ref{4.7}$), one obtains
$$XY_i\ll HX^{1+\epsilon}Y_iX^{-2^{1-k}}q^{-2^{1-k}}.$$
This implies $$q\leq X^{1/(l+1)}.$$

We apply the triangle inequality and put $y=qy_1$ with $0\leq y_1\leq X^{1-(i+1)/(l+1)}$. Thus, we have $qy_1\leq Y_i$ and 
\begin{equation}\label{4.8}
\begin{aligned}
&\|\alpha_k x^k+\alpha_{l-i}^{(i)}x^{l-i}(qy_1)^{k-l+i}+\cdots+ \alpha_0^{(i)} (qy_1)^k\|\\
& \leq \|\alpha_kx^k+\alpha_{l-i-1}^{(i)}x^{l-i-1}(qy_1)^{k-l+i+1}+\cdots+\alpha_0^{(i)}(qy_1)^k\|+\|\alpha_{l-i}^{(i)} x^{l-i}(qy_1)^{k-l+i}\|.
  \end{aligned}
\end{equation}
Since we have $H=X^{\sigma-\epsilon}$ with $\sigma=\frac{l+2}{l+1}2^{1-k},$ by applying the triangle inequality and ($\ref{4.5}$), one has
\begin{equation}\label{4.9}
    \begin{aligned}
        \|\alpha_{l-i}^{(i)} x^{l-i}(qy_1)^{k-l+i}\|&\leq x^{l-i}(qy_1)^{k-l+i-1}y_1\|\alpha_{l-i}^{(i)}q\|\\
        &\leq X^{l-i}Y_i^{k-l+i-1}X^{1-(i+1)/(l+1)}\frac{X^{\delta}}{X^{l-i+1}Y_i^{k-l+i}H^{-2^{k-1}+1}}\\
        &\leq \frac{X^{\epsilon}}{HX}H^{2^{k-1}}X^{-1/(l+1)}\\
        &\leq \frac{X^{-\epsilon}}{H}.
    \end{aligned}
\end{equation}
Hence, by substituting the bound ($\ref{4.9}$) into ($\ref{4.8}$) and the lower bound ($\ref{4.77}$), one obtains
\begin{equation}\label{4.10}
    \frac{1}{H}\ll \|\alpha_kx^k+\alpha_{l-i-1}^{(i)}x^{l-i-1}(qy_1)^{k-l+i+1}+\cdots+\alpha_0^{(i)}(qy_1)^k\|
\end{equation}
for all $0\leq x\leq X$ and $0\leq y_1\leq Y_{i+1}$ but $(x,y_1)=(0,0).$ By writing $$\alpha_{l-i-r}^{(i+1)}=\alpha_{l-i-r}^{(i)}q^{k-l+i+r}\ (1\leq r\leq l-i),$$
one concludes from ($\ref{4.10}$) that 
\begin{equation}\label{4.111}
     \frac{1}{H}\ll \|\alpha_kx^k+\alpha_{l-i-1}^{(i+1)}x^{l-i-1}y_1^{k-l+i+1}+\cdots+\alpha_0^{(i+1)}y_1^k\|
\end{equation}
for all $0\leq x\leq X$ and $0\leq y_1\leq Y_{i+1}$ but $(x,y)=(0,0).$

Thus, one infers from induction that
$$\frac{1}{H}\ll \|\alpha_k x^k+\alpha_0^{(l)}y^k\|$$
for all $0\leq x\leq X$ and $0\leq y\leq X^{1/(l+1)}$ but $(x,y)=(0,0),$
which contradicts the conclusion of Proposition 4.1. Thus, we are forced to conclude that 
$$\min_{\substack{0\leq x,y\leq X\\ (x,y)\neq (0,0)}}\|\alpha_kx^k+\alpha_lx^ly^{k-l}+\alpha_{l-1}x^{l-1}y^{k-l+1}+\cdots+\alpha_0y^k\|\leq 1/H.$$\end{proof}

\begin{proof}[Proof of Corollary 1.2]
By change of variables,
\begin{equation}\label{eq4.13}
\begin{aligned}
   & x= x_1-c_{k-1}y_1   \\
   & y= kc_{k}y_1,
\end{aligned}                  
\end{equation}
one has
\begin{align}\label{eq4.14}
&\varphi(x,y)=\dsum_{0\leq i\leq k} c_ix^{i}y^{k-i}=\dsum_{0\leq i\leq k}c_i(x_1-c_{k-1}y_1)^i(kc_{k}y_1)^{k-i}=\dsum_{\substack{0\leq i\leq k\\ i\neq k-1}} c_i' x_1^{i}y_1^{k-i},
\end{align}
where each $c_i'$ is obtained by the Binomial theorem.
Furthermore, notice from ($\ref{eq4.13}$) that 
$$0\leq x_1\leq X/2\ \textrm{and}\ 0\leq y_1\leq Y=\frac{X}{k(1+|c_k|)(1+|c_{k-1}|)}$$
implies that
$$ |x|\leq X\ \textrm{and}\ |y|\leq X.$$
Hence, 
 one has
\begin{equation*}
\begin{aligned}
    & \min_{\substack{0\leq |x|,|y|\leq X\\(x,y)\neq (0,0)}}\|\alpha_k \varphi(x,y)+\alpha_{k-2} x^{k-2}y^2+\cdots+\alpha_0 y^k\| \\
     \leq &\ \min_{\substack{0\leq x_1\leq X/2\\ 0\leq y_1\leq Y\\(x_1,y_1)\neq (0,0)}}\|\alpha_k x_1^k+\alpha_{k-2}' x_1^{k-2}y_1^2+\cdots+\alpha_0' y_1^k\|,
\end{aligned}
\end{equation*}
where each $\alpha_i'$ is obtained by ($\ref{eq4.14}$) and the Binomial theorem. Then, by applying Theorem 1.1 to the last expression, we conclude that for any $\epsilon>0$, there exists $X(k,\epsilon)$ such that whenever $X>X(k,\epsilon)$ one has
\begin{equation*}
    \min_{\substack{0\leq |x|,|y|\leq X\\(x,y)\neq (0,0)}}\|\alpha_k \varphi(x,y)+\alpha_{k-2} x^{k-2}y^2+\cdots+\alpha_0 y^k\| \leq X^{-\sigma+\epsilon},
\end{equation*}
where $\sigma=k/(k-1)2^{1-k}.$

\end{proof}

\section{Proof of Theorem 1.3}

Our goal in this section is to prove Theorem 1.3. We begin this section by examining the corresponding diagonal problem. The following proposition is useful for larger $k.$

\begin{prop}
Let $\alpha$, $\beta \in \R$ and  $k\in \N$ with $k\geq $. Then, for any $\epsilon>0$,  there exists a real number $X(k,\epsilon)$ such that whenever $X\geq X(k,\epsilon)$ and $Y\geq X(k,\epsilon)$ one has
\begin{equation}\label{50}
    \min_{\substack{0\leq x\leq X\\0\leq y\leq Y\\(x,y)\neq(0,0)}}\|\alpha x^k+\beta y^k\|\leq X^{-\sigma+\epsilon}Y^{-\sigma+\epsilon},
\end{equation}
where $\sigma=1/(k(k-1)).$
\end{prop}

\begin{proof}
Let $H=X^{\sigma-\epsilon}Y^{\sigma-\epsilon}$ with $\sigma=1/(k(k-1)).$ Suppose that we have
\begin{equation}\label{5.22}
\|\alpha x^k+\beta y^k\|>1/H    
\end{equation}
for all $x$ and $y$. By Dirichlet's approximation theorem, there exist $q_1,q_2\in \N$ and $a_1,a_2\in \Z$ with $(q_1,a_1)=(q_2,a_2)=1$ such that $q_1,q_2\leq X^{k-1}H$ and 
\begin{equation*}
    \bigl|\alpha-\frac{a_1}{q_1}\bigr|\leq \frac{1}{q_1X^{k-1}H},\ \bigl|\beta-\frac{a_2}{q_2}\bigr|\leq \frac{1}{q_2X^{k-1}H}.
\end{equation*}
By the same treatment in the proof of Proposition 4.1, we may assume that $q_1$ and $q_2$ are greater than $X.$

By Dirichlet's approximation theorem, there exists $r_1\in \N$ and $b_1\in\Z$ with $(r_1,b_1)=1$ such that $|h\alpha-b_1/r_1|\leq r_1^{-1}X^{1-k}$ and $r_1\leq X^{k-1}.$ Thus, by [$\ref{ref11}$, Theorem 5.2] depending on Vinogradov's man value theorem, one has
$$\dsum_{1\leq x\leq X}e(h\alpha x^k) \ll X^{1+\epsilon/3}\biggl(\frac{1}{r_1}+\frac{1}{X}+\frac{r_1}{X^k}\biggr)^{\sigma},$$
where $\sigma=1/(k(k-1)).$
Then, by the transference principle [$\ref{ref28}$, Lemma 14.1], one obtains
$$\dsum_{1\leq x\leq X}e(h\alpha x^k) \ll X^{1+\epsilon/3}\biggl(\frac{1}{r_1+X^k\|r_1h\alpha\|}+\frac{1}{X}+\frac{r_1+X^k\|r_1h\alpha\|}{X^k}\biggr)^{\sigma}.$$
Since $r_1\leq X^{k-1}$, this bound is seen to be 
\begin{equation}\label{eq5.2}
    \dsum_{1\leq x\leq X}e(h\alpha x^k)\ll X^{1-\sigma+\epsilon/3}+X^{1+\epsilon/3}\bigl(r_1+X^k\|r_1h\alpha\|\bigr)^{-\sigma}.
\end{equation}
Similarly, whenever $|h\beta-b_2/r_2|\leq r_2^{-1}Y^{1-k}$ and $r_2\leq Y^{k-1}$ with $(r_2,b_2)=1,$ one obtains 
\begin{equation}\label{eq5.3}
    \dsum_{1\leq y\leq y}e(h\beta y^k)\ll Y^{1-\sigma+\epsilon/3}+Y^{1+\epsilon/3}\bigl(r_2+Y^k\|r_2h\beta\|\bigr)^{-\sigma}.
\end{equation}
Notice here that $r_1,\ r_2$ may depend on $h.$
Meanwhile, by Lemma 2.1, it follows from our assumption $(\ref{5.22})$ that 
\begin{equation}\label{eq5.4}
XY\ll \dsum_{1\leq h\leq H}\biggl|\dsum_{\substack{1\leq x\leq X\\1\leq y\leq Y}}e(h(\alpha x^k+\beta y^k))\biggr|.    
\end{equation}
Then, by $(\ref{eq5.2})$ and $(\ref{eq5.3})$, we have
\begin{align*}
    \dsum_{1\leq h\leq H}\biggl|\dsum_{\substack{1\leq x\leq X\\1\leq y\leq Y}}e(h(\alpha x^k+\beta y^k))\biggr|&=  \dsum_{1\leq h\leq H}\biggl|\dsum_{\substack{1\leq x\leq X}}e(h(\alpha x^k))\biggr|\biggl|\dsum_{1\leq y\leq Y}e(h\beta y^k)\biggr|\\
    &\ll \dsum_{1\leq h\leq H}(U_1(h)+U_2(h)+U_3(h)+U_4(h)),
\end{align*}
where 
\begin{align*}
&U_1(h)=X^{1-\sigma+\epsilon/3}Y^{1-\sigma+\epsilon/3},\\
&U_2(h)=X^{1-\sigma+\epsilon/3}Y^{1+\epsilon/3}(r_2+Y^k\|r_2h\beta\|)^{-\sigma},\\ &U_3(h)=X^{1+\epsilon/3}Y^{1-\sigma+\epsilon/3}(r_1+X^k\|r_1h\alpha\|)^{-\sigma},\\ &U_4(h)=X^{1+\epsilon/3}Y^{1+\epsilon/3}(r_1+X^k\|r_1h\alpha\|)^{-\sigma}(r_2+Y^k\|r_2h\beta\|)^{-\sigma}
\end{align*}
in which $r_1$ and $r_2$ depend on $h.$

We shall show that 
$$\dsum_{1\leq h\leq H}(U_1(h)+U_2(h)+U_3(h)+U_4(h))\ll (XY)^{1-\epsilon/2}.$$

First, since $H=X^{\sigma-\epsilon}Y^{\sigma-\epsilon},$ we have 
\begin{equation}\label{eq5.55}
\dsum_{1\leq h\leq H}U_1(h)\ll X^{1-\epsilon/2}Y^{1-\epsilon/2}.    
\end{equation}

Second, consider $\dsum_{1\leq h\leq H}U_2(h).$ By H\"older's inequality, one has 
\begin{equation}\label{eq5.5}
\begin{aligned}
\dsum_{1\leq h\leq H}U_2(h)&\leq X^{1-\sigma+\epsilon/3}Y^{1+\epsilon/3}H^{1-\sigma}\left(\dsum_{1\leq h\leq H}(r_2+Y^k\|r_2h\beta\|)^{-1}\right)^{\sigma}\\
&=X^{1-\sigma+\epsilon/3}Y^{1+\epsilon/3}H^{1-\sigma}\left(Y^{-k}\dsum_{1\leq h\leq H}\min\left\{\frac{Y^k}{r_2},\|r_2h\beta\|^{-1}\right\}\right)^{\sigma}.
\end{aligned}
\end{equation}
Note that since $r_2\leq Y^{k-1}$, one has 
\begin{align*}
    \dsum_{1\leq h\leq H}\min\left\{\frac{Y^k}{r_2},\|r_2h\beta\|\right\}&\leq \dsum_{1\leq r\leq Y^{k-1}}\dsum_{1\leq h\leq H} \min\left\{\frac{Y^k}{r},\|rh\beta\|^{-1}\right\}\\
    &\leq \dsum_{1\leq r\leq Y^{k-1}}\dsum_{1\leq h\leq H} \min\left\{\frac{Y^kH}{rh},\|rh\beta\|^{-1}\right\}.
\end{align*}
By a standard divisor estimate, this bound is seen to be 
\begin{align*}
    \ll X^{\epsilon}\dsum_{1\leq n\leq Y^{k-1}H}\min\left\{\frac{Y^kH}{n},\|n\beta\|\right\}.
\end{align*}
By [$\ref{ref11}$, Lemma 2.2], this bound is
\begin{equation}\label{eq5.6}
    \ll X^{\epsilon}Y^kH\left(\frac{1}{q_2}+\frac{1}{Y}+\frac{q_2}{Y^kH}\right)\leq X^{\epsilon}Y^{k-1}H,
\end{equation}
where we have used inequalities $|\beta-a_2/q_2|\leq q_2^{-1}X^{1-k}H^{-1}$ and $q_2>X.$ Hence, on substituting ($\ref{eq5.6}$) into ($\ref{eq5.5}$), we find that 
\begin{equation}\label{eq5.7}
   \dsum_{1\leq h\leq H}U_2(h)\ll X^{1-\sigma+\epsilon/2}Y^{1-\sigma+\epsilon/2}H\ll X^{1-\epsilon/2}Y^{1-\epsilon/2}.
\end{equation}
 
 Third, consider $\dsum_{1\leq h\leq H}U_3(h)$. By the same treatment with just above, we obtain
 \begin{equation}\label{eq5.99}
 \begin{aligned}
\dsum_{1\leq h\leq H}U_3(h)&= X^{1+\epsilon/3}Y^{1-\sigma+\epsilon/3}H^{1-\sigma}\left(X^{-k}\dsum_{1\leq h\leq H}\min\left\{\frac{X^k}{r_1},\|r_1h\alpha\|^{-1}\right\}\right)^{\sigma}\\
&\ll X^{1+\epsilon/3}Y^{1-\sigma+\epsilon/3}H^{1-\sigma}\left(X^{\epsilon}H\left(\frac{1}{q_1}+\frac{1}{X}+\frac{q_1}{X^kH}\right)\right)^{\sigma}\\
&\ll X^{1-\epsilon/2}Y^{1-\epsilon/2}.
\end{aligned}
 \end{equation}

Finally, consider $\dsum_{1\leq h\leq H}U_4(h).$ By H\"older's inequality, we have
\begin{equation}\label{eq5.9}
    \dsum_{1\leq h\leq H}U_4(h)\ll (XY)^{1+\epsilon/3}H^{1-2\sigma}A^{\sigma}B^{\sigma},
\end{equation}
where
$$A=X^{-k}\dsum_{1\leq h\leq H}\min\left\{\frac{X^{k}}{r_1}, \|r_1h\alpha\|^{-1}\right\}$$ and
$$B=Y^{-k}\dsum_{1\leq h\leq H}\min\left\{\frac{Y^k}{r_2},\|r_2h\beta\|^{-1}\right\}.$$
By the same treatment in the case $ \dsum_{1\leq h\leq H}U_2(h)$ and $\dsum_{1\leq h\leq H}U_3(h)$, one infers from ($\ref{eq5.9}$) that 
\begin{equation}\label{eq5.10}
\dsum_{1\leq h\leq H}U_4(h)\ll (XY)^{1+\epsilon/3}H^{1-2\sigma}(Y^{\epsilon}HY^{-1})^{\sigma}(X^{\epsilon}HX^{-1})^{\sigma}\ll(XY)^{1-\epsilon/2}.  
\end{equation}

Hence, by ($\ref{eq5.55}$), ($\ref{eq5.7}$), ($\ref{eq5.99}$) and ($\ref{eq5.10}$), we have 
$$\dsum_{1\leq h\leq H}(U_1(h)+U_2(h)+U_3(h)+U_4(h))\ll (XY)^{1-\epsilon/2}.$$
This contradicts $(\ref{eq5.4})$ stemming from our assumption that 
$$\min_{\substack{0\leq x\leq X\\0\leq y\leq Y\\(x,y)\neq (0,0)}}\|\alpha x^k +\beta y^k\|> 1/H.$$
Therefore, we are forced to conclude that
$$\min_{\substack{0\leq x\leq X\\0\leq y\leq Y\\(x,y)\neq (0,0)}}\|\alpha x^k +\beta y^k\|\leq 1/H.$$
\end{proof}

We shall reduce the problem in Theorem 1.3 to this diagonal problem by exploiting the same argument in Theorem 1.1 with Lemma 3.1 replaced by Lemma 3.2.
\begin{proof}[Proof of Theorem 1.3]

Let $H=X^{\sigma-\epsilon}$ with $\sigma=\frac{2}{k(k-1)+\rho(k,l)}.$ 
Suppose that 
\begin{equation}\label{5.8}
    1/H < \|\alpha_kx^k+\alpha_lx^ly^{k-l}+\cdots+\alpha_0 y^k\|
\end{equation}
for all $1\leq x,y\leq X.$ From ($\ref{5.8}$), we shall derive a lower bound of fractional parts of polynomial having fewer terms by an inductive arguments. Specifically, at the $i$-th step, with $0\leq i\leq l,$ we shall show that for all $1\leq x\leq X$ and $1\leq y\leq Y_i$ with 
\begin{equation}\label{eqeq5.11}
Y_i=X^{1-(2^{l+1}(k-l)+2^l(k-l+1)+\cdots+2^{l-i+2}(k-l+i-1))\sigma},    
\end{equation}
there exist coefficients $\alpha_{l-i}^{(i)},\ldots,\alpha_0^{(i)}$, suitable integral multiples of $\alpha_{l-i},\ldots,\alpha_0$, respectively, such that 
\begin{equation}\label{eq5.11}
    \frac{1}{H}\ll \|\alpha_kx^k+\alpha_{l-i}^{(i)}x^{l-i}y^{k-l+i}+\cdots+\alpha_0^{(i)}y^k\|.
\end{equation}
The case $i=0$ obviously follows from ($\ref{5.8}$). Assume that ($\ref{eq5.11}$) holds for a particular $i\leq l-1.$ We shall show that ($\ref{eq5.11}$) holds for $i$ replaced by $i+1.$

Since $\mathcal{A}(Y_i,R)\subseteq [1,Y_i]\cap \Z$, the assumption ($\ref{eq5.11}$) implies
\begin{equation}
    \frac{1}{H}\ll \min_{\substack{1\leq x\leq X\\y\in \mathcal{A}(Y_i,R)}}\|\alpha_k x^k+\alpha_{l-i}^{(i)} x^{l-i}y^{k-l+i}+\cdots+\alpha_0^{(i)}y^k\|.
\end{equation}
Thus, by Lemma 2.1 and the fact that $|\mathcal{A}(Y_i,R)|\asymp Y_i,$ one has 
\begin{equation}\label{eq5.14}
    XY_i\ll \dsum_{1\leq h\leq H}\biggl|\dsum_{\substack{1\leq x\leq X\\y\in \mathcal{A}(Y_i,R)}}e(h(\alpha_kx^k+\alpha_{l-i}^{(i)}x^{l-i}y^{k-l+i}+\cdots+\alpha_0^{(i)}y^k))\biggr|
\end{equation}
By Dirichlet's theorem, there exist $q\in \N,a\in \Z$ with $(q,a)=1$ such that $$q\leq HX^{l-i-2^{l-i+1}(k-l+i)\sigma+\eta}Y_i^{k-l+i}$$ and 
\begin{equation}\label{eq5.15}
    \|q\alpha_{l-i}^{(i)}\|\leq \frac{X^{-\eta}X^{2^{l-i+1}(k-l+i)\sigma}}{HX^{l-i}Y_i^{k-l+i}}.
\end{equation}
Let $U=X^{2^{l-i+2}\sigma}.$ To apply Lemma 3.2 with $Z=U$ and $Y=Y_i$, we verify here that $U$ and $Y_i$ satisfy the hypotheses of Lemma 3.2, namely,
\begin{equation*}
    U\leq \min\{(HX^{l-i})^{1/(2(k-l+i))}, Y_i\}.
\end{equation*}

We first verify that $U\leq Y_i.$ Since we have $\rho(k,l)\leq k(k-1)$, one has 
$2^{l+1}(k-l)\leq k(k-1),$ and thus  
\begin{equation}\label{eqeq5.15}
2^l\leq k(k-1)/(2(k-l))\leq k-1,    
\end{equation}
since by an obvious restriction $l\leq\lfloor\frac{k}{2}\rfloor.$
Note also that
\begin{equation}\label{eqeq5.16}
\sigma=2/(k(k-1)+\rho(k,l))\leq \rho(k,l)^{-1},    
\end{equation} 
 since $\rho(k,l)\leq k(k-1).$ Hence, on recalling ($\ref{eqeq5.11}$), by $(\ref{eqeq5.15})$ and $(\ref{eqeq5.16})$, one finds that 
\begin{equation}\label{eqeq5.18}
    \begin{aligned}
       U=X^{2^{l-i+2}\sigma}\leq X^{2^{l+2}\rho(k,l)^{-1}}\leq X^{1-(\rho(k,l)-2^2(k-1))\rho(k,l)^{-1}}\leq X^{1-(\rho(k,l)-2^2(k-1))\sigma}\leq Y_i.
    \end{aligned}
\end{equation}

We turn to verify the other hypothesis. Since we have $\rho(k,l)\leq k(k-1)/7$ and by recalling the definition of $\rho(k,l)$, one has 
\begin{equation*}
    \sigma=\frac{2}{k(k-1)+\rho(k,l)}\leq \frac{1}{4\rho(k,l)}\leq \frac{1}{2^{l-i+3}(k-l+i)}.
\end{equation*}
Hence, we have 
\begin{equation}\label{eqeq5.22}
    U=X^{2^{l-i+2}\sigma}\leq X^{1/(2(k-l+i))}
\end{equation}

Thus, from ($\ref{eqeq5.18}$) and ($\ref{eqeq5.22}$), we verified that
\begin{equation*}
    U\leq \min\{(HX^{l-i})^{1/(2(k-l+i))}, Y_i\}.
\end{equation*}
Therefore, by applying Lemma 3.2 to the right hand side on $(\ref{eq5.14})$, we obtain 
\begin{equation}
    XY_i\ll \frac{H(XY_i)^{1+\epsilon}}{(q+HX^{l-i}Y_i^{k-l+i}|q\alpha_l-a|)^{1/(2^{l-i+1}(k-l+i))}}+H(XY_i)^{1+\epsilon}U^{-1/2^{l-i+2}}.
\end{equation}
This implies 
\begin{equation}
    q\leq X^{2^{l-i+1}(k-l+i)\sigma}.
\end{equation}
By applying the triangle inequality and putting $y=qy_1$ with $$1\leq y_1\leq Y_iX^{-2^{l-i+1}(k-l+i)\sigma}=Y_{i+1},$$
one finds that 
\begin{equation}\label{eq5.30}
    \begin{aligned}
      & \|\alpha_k x^k+\alpha_{l-i}^{(i)} x^{l-i}y^{k-l+i}+\cdots+\alpha_0^{(i)}y^k\|\\
       &\leq \|\alpha_k x^k+\alpha_{l-i-1}^{(i)} x^{l-i-1}(qy_1)^{k-l+i+1}+\cdots+\alpha_0^{(i)}(qy_1)^k\|+\|\alpha_{l-i}^{(i)}x^{l-i}(qy_1)^{k-l+i}\|.
    \end{aligned}
\end{equation}
By applying the triangle inequality and ($\ref{eq5.15}$), one has
\begin{equation}\label{eq5.31}
    \begin{aligned}
       \|\alpha_{l-i}^{(i)}x^{l-i}(qy_1)^{k-l+i}\|&\leq x^{l-i}(qy_1)^{k-l+i-1}y_1\|\alpha_{l-i}^{(i)}q\| \\
       &\leq X^{l-i}Y_{i}^{k-l+i}X^{-2^{l-i+1}(k-l+i)\sigma}\frac{X^{-\eta}X^{2^{l-i+1}(k-l+i)\sigma}}{HX^{l-i}Y_i^{k-l+i}}\\
       &\leq \frac{X^{-\eta}}{H}.
    \end{aligned}
\end{equation}
Thus, by $(\ref{eq5.11}),(\ref{eq5.30})$ and $(\ref{eq5.31}),$ one has
\begin{equation}\label{eq5.32}
    1/H\ll \|\alpha_k x^k+\alpha_{l-i-1}^{(i)} x^{l-i-1}(qy_1)^{k-l+i+1}+\cdots+\alpha_0^{(i)}(qy_1)^k\|
\end{equation}
for all $1\leq x\leq X$ and $1\leq y_1\leq Y_{i+1}.$ By writing 
$$\alpha_{l-i-r}^{(i+1)}=\alpha_{l-i-r}^{(i)}q^{k-l+i+r}\ (1\leq r\leq l-i),$$ 
one concludes from ($\ref{eq5.32}$) that
$$ 1/H\ll \|\alpha_k x^k+\alpha_{l-i-1}^{(i+1)} x^{l-i-1}y_1^{k-l+i+1}+\cdots+\alpha_0^{(i+1)}y_1^k\|$$
for all $1\leq x\leq X$ and $1\leq y_1\leq Y_{i+1}.$ Thus, this confirms the inductive step.

Thus, one infers by induction that 
$$1/H\ll \|\alpha_kx^k+\alpha_0^{(l)}y^k\|$$
for all $1\leq x\leq X$ and $1\leq y\leq Y_l.$
This contradicts the conclusion of Proposition 5.1. Thus, we are forced to conclude that 
$$\min_{\substack{0\leq x,y\leq X\\ (x,y)\neq (0,0)}}\|\alpha_kx^k+\alpha_lx^ly^{k-l}+\alpha_{l-1}x^{l-1}y^{k-l+1}+\cdots+\alpha_0y^k\|\leq 1/H.$$\end{proof}

\section{Proof of Theorem 1.4}
Our goal in this section is to prove Theorem 1.4. We begin this section by examining the corresponding diagonal problem. To prove following Proposition 6.2, we require the minor arc estimates in [$\ref{ref12}$, Corollary 2]. Throughout this section, we take $R=X^{\eta}$ with $\eta$ positive sufficient small. We state here this corollary without proof as a proposition.

\begin{prop}
Let $\mathfrak{m}_{\lambda}$ denote the set of $\alpha\in \R$ such that whenever $a\in \Z,$ $q\in \N,$ $(a,q)=1$ and $|\alpha-a/q|\leq q^{-1}X^{\lambda-k},$ then $q>X^{\lambda}R$. Then there is a natural number $k_0(\epsilon)$ with the following property. When $k\geq k_0(\epsilon),$ there are real numbers $\lambda=\lambda(k)$, $\sigma(k)$ and $C>0$ with 
$$\frac{\log\log k}{\log k}\ll 1-\lambda\ll \frac{\log\log k}{\log k}\ \ \textrm{and}\ \ \sigma(k)^{-1}=k(\log k+C\log\log k),$$
and such that 
$$\dsum_{x\in\mathcal{A}(X,R)}e(\alpha x^k)\ll X^{1-\sigma(k)+\epsilon}.$$
\end{prop}
\begin{proof}
See [$\ref{ref12}$, Corollary 2].
\end{proof}

To state the following proposition, we exploit the exponent $\lambda=\lambda(k)$ defined in Proposition 6.1.

\begin{prop}
Let $\alpha,\beta \in \R$ and  $k\in \N$. Define $$\sigma=\frac{1}{k\log k+Ck\log\log k},\ \ \ \sigma_0=\frac{\sigma}{1-\lambda-\sigma}.$$ Let $X$ and $Y$ be  real numbers sufficiently large in terms of $k$ and $\epsilon$ with $X^{\sigma_0}\leq Y\leq X.$ Then, one has
\begin{equation*}
    \min_{\substack{0\leq x\leq X\\0\leq y\leq Y\\(x,y)\neq(0,0)}}\|\alpha x^k+\beta y^k\|\leq  X^{-\sigma+\epsilon}Y^{-\sigma+\epsilon}.
\end{equation*}

\end{prop}
\begin{proof}
Let $H=X^{\sigma-\epsilon}Y^{\sigma-\epsilon}.$ Notice that $\max\{X^{\lambda-1}, Y^{\lambda-1}\}\leq 1/H.$ Suppose that there exists no $x,y$ satisfying $\|\alpha x^k+\beta y^k\|\leq 1/H.$

Suppose first that there exist $q\in \N,\ a\in \Z,\ h\in [1,H]\cap \Z$ with $(q,a)=1$ such that $q\leq X^{\lambda}R$ and 
\begin{equation}\label{62}
    \bigl|h\alpha -\frac{a}{q}\bigr|\leq q^{-1}X^{\lambda-k}.
\end{equation}
Since we have $Y\leq X$, we find that $hq\leq HX^{\lambda}R\leq X$. Then, by putting $y=qh$, it follows from $(\ref{62})$ that
\begin{equation*}
    \min_{\substack{0\leq x\leq X\\0\leq y\leq Y\\(x,y)\neq(0,0)}}\|\alpha x^k+\beta y^k\|\leq \|\alpha (qh)^k\|\leq (HX^{\lambda}R)^{k-1}\|\alpha qh\|\leq X^{k-1}X^{\lambda-k}\leq 1/H,
\end{equation*}
which contradicts our assumption.

Suppose next that there exist $q\in \N,\ a\in \Z,\ h\in [1,H]\cap \Z$ with $(q,a)=1$ such that $q\leq Y^{\lambda}R$ and 
\begin{equation}\label{63}
    \bigl|h\beta-\frac{a}{q}\bigr|\leq Y^{\lambda-k}.
\end{equation}
Since we have $X^{\sigma_0}\leq Y,$ we find that $qh\leq HY^{\lambda}R\leq Y.$  Then, by putting $y=qh$, it follows from ($\ref{63}$) that 
\begin{equation*}
    \min_{\substack{0\leq x\leq X\\0\leq y\leq Y\\(x,y)\neq(0,0)}}\|\alpha x^k+\gamma y^k\|\leq \|\gamma (qh)^k\|\leq (HY^{\lambda}R)^{k-1}\|\gamma qh\|\leq Y^{k-1}Y^{\lambda-k}\leq 1/H,
\end{equation*}
which contradicts our assumption. 

Thus, on recalling the definition of $\mathfrak{m}_{\lambda}$ in the statement of proposition 6.1. We may assume that for all $h\in [1,H]\cap \Z,$ one has $h\alpha, h\beta\in \mathfrak{m}_{\lambda}.$ Meanwhile, it follows from our assumption by Lemma 2.1 that 
\begin{equation*}
    XY\ll \dsum_{1\leq h\leq H}\bigl|\dsum_{\substack{x\in \mathcal{A}(X,R)\\y\in \mathcal{A}(Y,R)}}e(h(\alpha x^k+\beta y^k))\bigr|.
\end{equation*}
Therefore, on writing 
\begin{equation*}
    \bigl|\dsum_{\substack{x\in \mathcal{A}(X,R)\\y\in \mathcal{A}(Y,R)}}e(h(\alpha x^k+\beta y^k))\bigr|=\bigl|\dsum_{x\in \mathcal{A}(X,R)}e(h(\alpha x^k))\bigr|\bigl|\dsum_{y\in \mathcal{A}(Y,R)}e(h(\beta y^k))\bigr|,
\end{equation*}
by applying Proposition 6.1, we find that
\begin{equation*}
    \dsum_{1\leq h\leq H}\bigl|\dsum_{\substack{x\in \mathcal{A}(X,R)\\y\in \mathcal{A}(Y,R)}}e(h(\alpha x^k+\beta y^k))\bigr|\ll (XY)^{1-\epsilon},
\end{equation*}
which contradicts our assumption. Hence, we are forced to have 
\begin{equation*}
    \min_{\substack{0\leq x\leq X\\0\leq y\leq Y\\(x,y)\neq(0,0)}}\|\alpha x^k+\beta y^k\|\leq  1/H.
\end{equation*}
\end{proof}

\begin{proof}[Proof fo Theorem 1.4]
Let $H=X^{\sigma-\epsilon}$ with $\sigma=2/(k\log k+\rho(k,l)+Ck\log\log k)$. Suppose that 
\begin{equation}\label{eq6.4}
    1/H\leq \|\alpha_kx^k+\alpha_lx^ly^{k-l}+\cdots+\alpha_0y^k\|
\end{equation}
for all $1\leq x,y\leq X.$ From ($\ref{eq6.4}$), we shall derive a lower bound of fractional parts of polynomial having fewer terms by an inductive arguments. Specifically, at the $i$-th step, with $0\leq i\leq l,$ we shall show that for all $1\leq x\leq X$ and $1\leq y\leq Y_i$ with 
\begin{equation*}
Y_i=X^{1-(2^{l+1}(k-l)+2^l(k-l+1)+\cdots+2^{l-i+2}(k-l+i-1))\sigma},    
\end{equation*}
there exist coefficients $\alpha_{l-i}^{(i)},\ldots,\alpha_0^{(i)}$, suitable integral multiples of $\alpha_{l-i},\ldots,\alpha_0$, respectively, such that 
\begin{equation}\label{eq6.6}
    \frac{1}{H}\ll \|\alpha_kx^k+\alpha_{l-i}^{(i)}x^{l-i}y^{k-l+i}+\cdots+\alpha_0^{(i)}y^k\|.
\end{equation}
The case $i=0$ obviously follows from ($\ref{eq6.4}$). Assume that ($\ref{eq6.6}$) holds for a particular $i\leq l-1.$ We shall show that ($\ref{eq6.6}$) holds for $i$ replaced by $i+1.$

Since $\mathcal{A}(Y_i,R)\subseteq [1,Y_i]\cap \Z,$ the assumption ($\ref{eq6.6}$) implies
\begin{equation}
    \frac{1}{H}\ll \min_{\substack{1\leq x\leq X\\y\in \mathcal{A}(Y_i,R)}}\|\alpha_k x^k+\alpha_{l-i}^{(i)} x^{l-i}y^{k-l+i}+\cdots+\alpha_0^{(i)}y^k\|.
\end{equation}
Thus, by Lemma 2.1 and the fact that $|\mathcal{A}(Y_i,R)|\asymp Y_i,$ one has 
\begin{equation}\label{eq6.8}
    XY_i\ll \dsum_{1\leq h\leq H}\biggl|\dsum_{\substack{1\leq x\leq X\\y\in \mathcal{A}(Y_i,R)}}e(h(\alpha_kx^k+\alpha_{l-i}^{(i)}x^{l-i}y^{k-l+i}+\cdots+\alpha_0^{(i)}y^k))\biggr|
\end{equation}
By Dirichlet's theorem, there exist $q\in \N,a\in \Z$ with $(q,a)=1$ such that $$q\leq HX^{l-i-2^{l-i+1}(k-l+i)\sigma+\eta}Y_i^{k-l+i}$$ and 
\begin{equation}\label{eq6.9}
    \|q\alpha_{l-i}^{(i)}\|\leq \frac{X^{-\eta}X^{2^{l-i+1}(k-l+i)\sigma}}{HX^{l-i}Y_i^{k-l+i}}.
\end{equation}
Let $U=X^{2^{l-i+2}\sigma}.$ By the same treatment in the proof of Theorem 1.3 with $k(k-1)$ replaced by $k\log k$, we find that this $U$ and $Y_i$ satisfy the hypotheses of Lemma 3.2. Therefore, by applying Lemma 3.2 to the right hand side on ($\ref{eq6.8}$), we obtain 
\begin{equation*}
    XY_i\ll \frac{H(XY_i)^{1+\epsilon}}{(q+HX^{l-i}Y_i^{k-l+i}|q\alpha_l-a|)^{1/(2^{l-i+1}(k-l+i))}}+H(XY_i)^{1+\epsilon}U^{-1/2^{l-i+2}}.
\end{equation*}
This implies 
\begin{equation}\label{eq6.10}
    q\leq X^{2^{l-i+1}(k-l+i)\sigma}.
\end{equation}
By applying the triangle inequality and putting $y=qy_1$ with $$1\leq y_1\leq Y_iX^{-2^{l-i+1}(k-l+i)\sigma}=Y_{i+1},$$
one finds that 
\begin{equation}\label{eq6.11}
    \begin{aligned}
      & \|\alpha_k x^k+\alpha_{l-i}^{(i)} x^{l-i}y^{k-l+i}+\cdots+\alpha_0^{(i)}y^k\|\\
       &\leq \|\alpha_k x^k+\alpha_{l-i-1}^{(i)} x^{l-i-1}(qy_1)^{k-l+i+1}+\cdots+\alpha_0^{(i)}(qy_1)^k\|+\|\alpha_{l-i}^{(i)}x^{l-i}(qy_1)^{k-l+i}\|.
    \end{aligned}
\end{equation}
By applying the triangle inequality and ($\ref{eq6.9}$), one has
\begin{equation}\label{eq6.12}
    \begin{aligned}
       \|\alpha_{l-i}^{(i)}x^{l-i}(qy_1)^{k-l+i}\|&\leq x^{l-i}(qy_1)^{k-l+i-1}y_1\|\alpha_{l-i}^{(i)}q\| \\
       &\leq X^{l-i}Y_{i}^{k-l+i}X^{-2^{l-i+1}(k-l+i)\sigma}\frac{X^{-\eta}X^{2^{l-i+1}(k-l+i)\sigma}}{HX^{l-i}Y_i^{k-l+i}}\\
       &\leq \frac{X^{-\eta}}{H}.
    \end{aligned}
\end{equation}
Thus, by $(\ref{eq6.6}),(\ref{eq6.11})$ and $(\ref{eq6.12}),$ one has
\begin{equation}\label{eq6.13}
    1/H\ll \|\alpha_k x^k+\alpha_{l-i-1}^{(i)} x^{l-i-1}(qy_1)^{k-l+i+1}+\cdots+\alpha_0^{(i)}(qy_1)^k\|
\end{equation}
for all $1\leq x\leq X$ and $1\leq y_1\leq Y_{i+1}.$ By writing 
$$\alpha_{l-i-r}^{(i+1)}=\alpha_{l-i-r}^{(i)}q^{k-l+i+r}\ (1\leq r\leq l-i),$$ 
one concludes from ($\ref{eq6.13}$) that
$$ 1/H\ll \|\alpha_k x^k+\alpha_{l-i-1}^{(i+1)} x^{l-i-1}y_1^{k-l+i+1}+\cdots+\alpha_0^{(i+1)}y_1^k\|$$
for all $1\leq x\leq X$ and $1\leq y_1\leq Y_{i+1}.$ 

Thus, one infers from induction that 
\begin{equation}\label{ineq6.14}
    1/H\ll \|\alpha_kx^k+\alpha_0^{(l)}y^k\|
\end{equation}
for all $1\leq x\leq X$ and $1\leq y\leq Y_l.$ Since $Y_l=X^{1-\rho(k,l)\sigma}$ and $\rho(k,l)\leq k\log k,$ a modicum  computation leads to the relation that $Y_l\geq X^{\sigma_0}.$ Therefore, the inequality ($\ref{ineq6.14}$) contradicts the conclusion of Proposition 6.2. Thus, we are forced to conclude that 
$$\min_{\substack{0\leq x,y\leq X\\ (x,y)\neq (0,0)}}\|\alpha_kx^k+\alpha_lx^ly^{k-l}+\alpha_{l-1}x^{l-1}y^{k-l+1}+\cdots+\alpha_0y^k\|\leq 1/H.$$
\end{proof}

\section*{Appendix A. Estimations in the proof of Proposition 3.4}
Let $\alpha,\ N$ be positive real numbers, and $k$ and $l$ be natural numbers with $k-l\geq 2$. Then, one has following:
\begin{align*}
&(1)\ \dsum_{1\leq n\leq N}\frac{1}{1+\alpha n^{k-l}}\ll \frac{N}{(1+N^{k-l}\alpha)^{1/(k-l)}}\\
&(2)\ \dsum_{1\leq n\leq N}\frac{1}{(1+\alpha n^{k-l})^{1/(k-l)}}\ll \frac{N\log N}{(1+N^{k-l}\alpha)^{1/(k-l)}}    \\
&(3)\ \dsum_{1\leq n\leq N}\frac{1}{(1+\alpha n)^{2/(k-l)}}\ll \frac{N\log N}{(1+N\alpha)^{2/(k-l)}}
\end{align*}
\begin{proof}[Proof of (1)]
We have 
$$\dsum_{1\leq n\leq N}\frac{1}{1+\alpha n^{k-l}}\ll \dsum_{1\leq n\leq N}\min\{1, \alpha^{-1}n^{-k+l}\}\ll N+\alpha^{-1/(k-l)},$$
where $N$ and $\alpha^{-1/(k-l)}$ are derived from the contributions $1$ and $\alpha^{-1}n^{-k+l}$, respectively.
\end{proof}
\begin{proof}[Proof of (2)]
We have 
$$\dsum_{1\leq n\leq N}\frac{1}{1+\alpha n^{k-l}}\ll \dsum_{1\leq n\leq N}\min\{1, \alpha^{-1/(k-l)}n^{-1}\}\ll N+\alpha^{-1/(k-l)}\log N,$$
where $N$ and $\alpha^{-1/(k-l)}\log N$ are derived from the contributions $1$ and $\alpha^{-1/(k-l)}n^{-1}$, respectively.
\end{proof}
\begin{proof}[Proof of (3)]
We have 
$$\dsum_{1\leq n\leq N}\frac{1}{1+\alpha n^{k-l}}\ll \dsum_{1\leq n\leq N}\min\{1, (\alpha n)^{-2/(k-l)}\}\ll N+\frac{N\log N}{\alpha^{2/(k-l)}N^{2/(k-l)}},$$
where $N$ and $\frac{N\log N}{\alpha^{2/(k-l)}N^{2/(k-l)}}$ are derived from the contributions $1$ and $(\alpha n)^{-2/(k-l)}$, respectively.
\end{proof}

\end{document}